\documentclass[11pt]{article}

\usepackage{subfig}
\usepackage{mathrsfs}
\usepackage{makecell}
\usepackage{amscd}
\usepackage{amsmath}
\usepackage{amstext}
\usepackage[thmmarks,amsthm,amsmath]{ntheorem}
\usepackage{apptools}
\usepackage{bbold}
\usepackage{bm}
\usepackage{booktabs}
\usepackage{color}
\usepackage{easybmat}
\usepackage{etex}
\usepackage{framed}
\usepackage[dvips,letterpaper,margin=1in]{geometry}
\usepackage{graphicx}
\usepackage{hyperref}
\usepackage[noabbrev,capitalize]{cleveref}
\usepackage{mathtools}
\usepackage{colonequals}
\usepackage{longtable}
\usepackage{rotating}
\usepackage{setspace}
\usepackage{tabu}
\usepackage{verbatim}

\newtheorem{theorem}{Theorem}[section]
\newtheorem{remark}[theorem]{Remark}
\newtheorem{lemma}[theorem]{Lemma}

\newtheorem{proposition}[theorem]{Proposition}
\newtheorem{definition}[theorem]{Definition}
\newtheorem{corollary}[theorem]{Corollary}
\AtAppendix{\renewtheorem{corollary}[theorem]{Corollary}}

\crefname{theorem}{Theorem}{Theorems}
\crefname{proposition}{Proposition}{Propositions}
\crefname{lemma}{Lemma}{Lemmas}

\theoremstyle{plain} 
\newcommand{\thistheoremname}{}
\newtheorem*{genericthm}{\thistheoremname}

\DeclareSymbolFont{bbold}{U}{bbold}{m}{n}
\DeclareSymbolFontAlphabet{\mathbbold}{bbold}

\newcommand{\Schlafli}{Schl\"{a}fli}

\newcommand{\ra}{\rangle}
\newcommand{\la}{\langle}

\makeatletter
\def\moverlay{\mathpalette\mov@rlay}
\def\mov@rlay#1#2{\leavevmode\vtop{%
   \baselineskip\z@skip \lineskiplimit-\maxdimen
   \ialign{\hfil$\m@th#1##$\hfil\cr#2\crcr}}}
\newcommand{\charfusion}[3][\mathord]{
    #1{\ifx#1\mathop\vphantom{#2}\fi
        \mathpalette\mov@rlay{#2\cr#3}
      }
    \ifx#1\mathop\expandafter\displaylimits\fi}
\makeatother

\renewcommand{\dim}{\mathsf{dim}}

\newcommand{\RR}{\mathbb{R}}

\newcommand{\NN}{\mathbb{N}}
\newcommand{\ZZ}{\mathbb{Z}}

\renewcommand{\SS}{\mathbb{S}}

\newcommand{\One}{\mathbbold{1}}

\newcommand{\bq}{\bm{q}}

\newcommand{\one}{\bm{1}}

\newcommand{\sA}{\mathcal{A}}
\newcommand{\sB}{\mathcal{B}}

\newcommand{\sO}{\mathcal{O}}

\newcommand{\sV}{\mathcal{V}}

\newcommand{\Tr}{\mathrm{Tr}}

\renewcommand{\ker}{\mathsf{ker}}
\newcommand{\sgn}{\mathsf{sgn}}

\newcommand{\sep}{\mathrm{sep}}

\newcommand{\eqqu}{\stackrel{\mathrm{?}}{=}}

\renewcommand{\vec}{\mathsf{vec}}
\renewcommand{\sep}{\mathsf{sep}}

\newcommand{\row}{\mathsf{row}}

\newcommand{\SOS}{\mathsf{SOS}}
\newcommand{\conv}{\mathsf{conv}}
\newcommand{\Gram}{\mathsf{Gram}}

\newcommand{\M}{\mathsf{M}}

\newcommand{\srg}{\mathsf{srg}}
\newcommand{\sym}{\mathsf{sym}}

\newcommand{\antisym}{\mathsf{antisym}}

\newcommand{\argmin}{\mathrm{argmin}}

\newcommand{\diag}{\mathsf{diag}}

\newcommand{\rank}{\mathsf{rank}}

\newcommand{\bA}{\bm A}

\newcommand{\bD}{\bm D}

\newcommand{\bM}{\bm M}

\newcommand{\bP}{\bm P}
\newcommand{\bQ}{\bm Q}
\newcommand{\bR}{\bm R}
\newcommand{\bS}{\bm S}

\newcommand{\bU}{\bm U}
\newcommand{\bV}{\bm V}
\newcommand{\bW}{\bm W}
\newcommand{\bX}{\bm X}
\newcommand{\bY}{\bm Y}
\newcommand{\bZ}{\bm Z}
\newcommand{\ba}{\bm a}
\newcommand{\bb}{\bm b}

\newcommand{\bd}{\bm d}
\newcommand{\be}{\bm e}

\newcommand{\bs}{\bm s}

\newcommand{\bt}{\bm t}

\newcommand{\bv}{\bm v}
\newcommand{\bw}{\bm w}
\newcommand{\bx}{\bm x}
\newcommand{\by}{\bm y}
\newcommand{\bz}{\bm z}

\newcommand{\fC}{\mathscr{C}}
\newcommand{\fE}{\mathscr{E}}

\newcommand{\odd}{\mathsf{odd}}

\newcommand{\GramSDP}{\mathsf{GramSDP}}
\newcommand{\ptop}{\mathsf{\Gamma}}

\title{A Gramian Description of the Degree 4 Generalized Elliptope}
\author{%
  Afonso S.\ Bandeira\footnote{Courant Institute of Mathematical Sciences and Center for Data Science, New York University, NY 10012. Afonso S. Bandeira was partially supported by NSF
grants DMS-1712730 and DMS-1719545, and by a grant from the Sloan
Foundation.} \and Dmitriy Kunisky\footnote{Courant Institute of Mathematical Sciences, New York University, NY 10012. Dmitriy Kunisky was partially supported by NSF
grants DMS-1712730 and DMS-1719545.}
}
\date{
  First Draft: December 30, 2018 \\
  Current Draft: March 24, 2019}

\begin{document}

\maketitle

\begin{abstract}
    One of the most widely studied convex relaxations in combinatorial optimization is  the relaxation of the cut polytope $\mathscr{C}^N$ to the elliptope $\mathscr{E}^N$, which corresponds to the degree 2 sum-of-squares (SOS) relaxation of optimizing a quadratic form over the hypercube $\{ \pm 1 \}^N$.
    We study the extension of this classical idea to degree 4 SOS, which gives an intermediate relaxation we call the \emph{degree 4 generalized elliptope} $\mathscr{E}_4^N$.
    Our main result is a necessary and sufficient condition for the Gram matrix of a collection of vectors to belong to $\mathscr{E}_4^N$.
    Consequences include a tight rank inequality between degree 2 and degree 4 pseudomoment matrices, and a guarantee that the only extreme points of $\mathscr{E}^N$ also in $\mathscr{E}_4^N$ are the cut matrices; that is, $\mathscr{E}^N$ and $\mathscr{E}_4^N$ share no ``spurious'' extreme point.

    For Gram matrices of equiangular tight frames, we give a simple criterion for membership in $\mathscr{E}_4^N$.
    This yields new inequalities satisfied in $\mathscr{E}_4^N$ but not $\mathscr{E}^N$ whose structure is related to the Schl\"{a}fli graph and which cannot be obtained as linear combinations of triangle inequalities.
    We also give a new proof of the restriction to degree 4 of a result of Laurent showing that $\mathscr{E}_4^N$ does not satisfy certain cut polytope inequalities capturing parity constraints.
    Though limited to this special case, our proof of the positive semidefiniteness of Laurent's pseudomoment matrix is short and elementary.

    Our techniques also suggest that membership in $\mathscr{E}_4^N$ is closely related to the partial transpose operation on block matrices, which has previously played an important role in the study of quantum entanglement.
    To illustrate, we present a correspondence between certain entangled bipartite quantum states and the matrices of $\mathscr{E}_4^N \setminus \mathscr{C}^N$.
\end{abstract}

\newpage

\tableofcontents

\newpage

\section{Introduction}

The optimization of quadratic forms over the hypercube $\{\pm 1\}^N$,
\begin{equation}
    \M(\bW) \colonequals \max_{\bx \in \{\pm 1\}^N} \bx^\top \bW \bx,
\end{equation}
is a well-studied computational problem arising in several contexts, including the Grothendieck problem \cite{grothendieck:56,khot:11}, graph problems such as finding the maximum cut \cite{goemans:95,poljak:95}, synchronization over the group $\ZZ / 2\ZZ$ \cite{abbe:14,bandeira:15}, spiked matrix estimation problems under priors with the i.i.d.\ Rademacher distribution \cite{alaoui:18,perry:18}, and the determination of ground state energies of hard two-spin models from statistical physics \cite{barahona:82,panchenko:13}.
$\M(\bW)$ may equivalently be viewed as optimizing a linear objective over a convex set called the \emph{cut polytope}:
\begin{align}
  \fC^N &\colonequals \conv(\{\bx\bx^\top: \bx \in \{\pm 1\}^N\}), \\
  \M(\bW) &= \max_{\bX \in \fC^N} \la \bW, \bX \ra.
\end{align}
Though it is convex, this problem is nonetheless difficult to solve exactly (e.g., NP-hard for $\bW$ a graph Laplacian, which computes the maximum cut \cite{karp:72}) due to the intricate discrete geometry of the cut polytope \cite{deza:09:book}.

A popular algorithmic choice for approximating $\M(\bW)$ and estimating its optimizer is to form \emph{relaxations} of $\fC^N$, larger convex sets admitting simpler descriptions.
Often, the relaxed sets may be described concisely in terms of positive semidefiniteness (psd) conditions, which leads to semidefinite programming (SDP) relaxations of $\M(\bW)$.
The most common way to execute this strategy is to optimize over the \emph{elliptope},
\begin{equation}
    \fE^N = \fE_2^N \colonequals \{\bX \in \RR^{N \times N}_{\sym}: \bX \succeq 0, \diag(\bX) = \one\} \supseteq \fC^N.
\end{equation}
For example, the well-known approximation algorithms of Goemans-Williamson~\cite{goemans:95} and Nesterov~\cite{nesterov:98} optimize over $\fE_2^N$ and then perform a \emph{rounding} procedure to recover an approximately optimal $\bx \in \{\pm 1\}^N$ from $\bX \in \fE_2^N$.

As our notation suggests, $\fE_2^N$ is only the first of a sequence of increasingly tighter relaxations of $\fC^N$, corresponding to \emph{sum-of-squares (SOS)} relaxations of $\M(\bW)$.
These sets are indexed by an even integer $d \geq 2$ called the \emph{degree}, and we call the set described at degree $d$ the \emph{degree $d$ generalized elliptope}, denoted $\fE_d^N$.
The generalized elliptopes satisfy the inclusions
\begin{equation}
    \fE_2^N \supseteq \fE_4^N \supseteq \cdots \supseteq \fE_{N + \One\{N \text{ odd}\}}^N = \fC^N.
\end{equation}
(The last equality and its tightness in the sense that no generalized elliptope of lower degree equals $\fC^N$ are non-trivial results proven by \cite{fawzi:16} and \cite{laurent:03} respectively.)
Thus, optimizing over generalized elliptopes of higher degree may yield better approximations of $\M(\bW)$; however, it is also costlier, since the associated semidefinite program is over matrix variables of size $N^{d / 2} \times N^{d / 2}$, whereby while general-purpose SDP algorithms will solve such an optimization to fixed accuracy in polynomial time in $N$, the bound on their runtime will be of order $N^{O(d)}$ \cite{ben:01}.
It is therefore important to know whether optimizing over generalized elliptopes of constant degree $d > 2$ actually improves the bounds on $\M(\bW)$ achieved by optimizing over $\fE_2^N$ on specific classes of optimization problems as $N \to \infty$.\footnote{There is an extensive literature relating this question to the Unique Games Conjecture \cite{khot:05,trevisan:12:survey}, which implies for several problems, most notably MaxCut, that optimizing over generalized elliptopes of constant degree cannot improve the worst-case approximation ratio achieved by optimizing over $\fE_2^N$ (see e.g.\ \cite{khot:07,raghavendra:08}). On the other hand, similar questions in the average case, for instance the typical quality of approximation that can be achieved for natural random models of $\bW$, are only beginning to be understood. For instance, \cite{montanari:16} is a recent work in this direction concerning SDP relaxations for MaxCut and related problems on random graphs.}

Nonetheless, although many fundamental geometric results on $\fE_2^N$ were obtained soon after this relaxation was introduced \cite{laurent:95,laurent:96}, relatively little remains known about the geometry of $\fE_{d}^N$ for specific constant values of $d > 2$.
Instead, most progress on higher-degree SOS relaxations has been through the algebraic interpretation of SOS, and concerning the limit $d \to \infty$ after $N \to \infty$.
A prominent line of work in this direction is on negative results for SOS, proving that its approximations are poor for some problems even at high degrees.
This began with the works of Grigoriev \cite{grigoriev:01:2,grigoriev:01,grigoriev:02} giving lower bounds for the SOS degree needed prove linear systems over $\ZZ / 2\ZZ$ unsatisfiable.
With similar techniques, Laurent \cite{laurent:03} produced examples of inequalities over $\fC^N$ that do not hold over $\fE_d^N$ until $d = \Omega(N)$.
Schoenebeck \cite{schoenebeck:08} later rediscovered Grigoriev's results and emphasized their application to constraint satisfaction problems.
More recently, similar ideas yielded rapid progress on the planted clique problem \cite{meka:15,deshpande:15,hopkins:18,barak:16}, culminating in a framework called \emph{pseudo-calibration} introduced in \cite{barak:16} for constructing pseudomoment matrices based on statistical reasoning, which has since been applied to several other problems \cite{hopkins:17,raghavendra:18}.
The general proof technique behind these results is to first construct candidate pseudomoment matrices entrywise, and then analyze whether those candidates satisfy the necessary constraints.
Among those constraints, the requirement that the pseudomoment matrix be psd is notoriously difficult to verify.

In this paper, we make two contributions to the state of affairs outlined above.
First, we derive some novel geometric facts about $\fE_4^N$, the first generalized elliptope, relating its extrema to the facial geometry of $\fE_2^N$.
Second, in doing so, we introduce a technique for constructing SOS pseudomoment matrices as Gram matrices of collections of vectors, which are therefore guaranteed by construction to be psd.
We use this to provide an alternate proof of part of Laurent's theorem \cite{laurent:03} mentioned above, and believe that this idea may eventually lead to simplified proofs of other difficult negative results in the SOS optimization literature.

\begin{figure}
    \centering
    \includegraphics[scale=0.55]{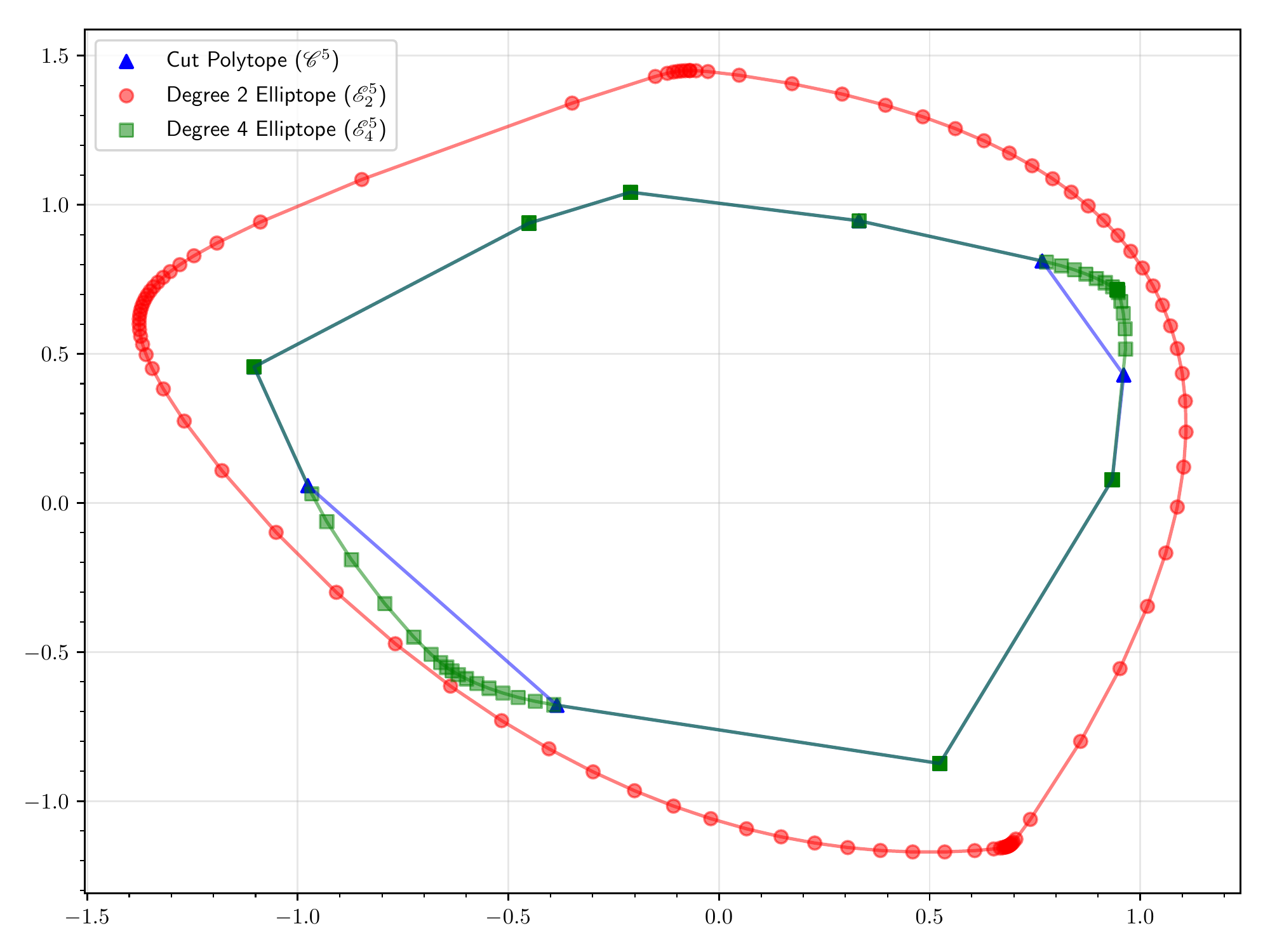}
    \caption{\textbf{Cut polytope and elliptopes in low dimension.} We plot the cross-section of $\fC^5$, $\fE_2^5$, and $\fE_4^5$ by an isotropic random subspace (in the off-diagonal matrix entries) by numerically solving suitably constrained linear and semidefinite programs. This is different from the projection of these sets onto such a subspace, for which we observe that $\fE_4^5$ is almost always indistinguishable from $\fC^5$. Note also that $N = 5$ is the lowest dimension where $\fE_4^N \neq \fC^N$.}
\end{figure}

\section{Main Results}

\subsection{Preliminaries}

\paragraph{Pseudomoment matrices.}
We first present the description of $\fE_d^N$ in terms of the \emph{pseudomoment} interpretation of SOS optimization.
The formal definition is as follows.

\begin{definition}
    For a finite set $\sA$, we write $\sA^{\leq d}$ for the collection of finite (possibly empty) strings in the elements of $\sA$ of length at most $d$, i.e.
    \begin{equation}
        \sA^{\leq d} \colonequals \{\emptyset\} \sqcup \sA \sqcup \sA^2 \sqcup \cdots \sqcup \sA^d.
    \end{equation}
    The sets $\sA^{\leq d}$ for all $d \in \NN \colonequals \{n \in \ZZ: n \geq 0\}$ may be thought of as embedded in the set $\sA^{< \infty}$ of all finite strings in the elements of $\sA$.
    For $\bs \in \sA^{< \infty}$, we write $|\bs|$ for the length of $\bs$, and for $\bs, \bt \in \sA^{< \infty}$, we write $\bs \circ \bt \in \sA^{|\bs| + |\bt|}$ for the concatenation of $\bs$ and $\bt$.
\end{definition}

\begin{definition}
    For $\bs \in \sA^{< \infty}$, $\odd(\bs) \subset \sA$ denotes the set of symbols that occur an odd number of times in $\bs$.
\end{definition}

\begin{definition}
    \label{def:truncated-pm}
    $\fE_d^N \subset \RR^{N \times N}_{\sym}$ is the set of $\bX$ such that there exists $\bY \in \RR^{N^{d / 2} \times N^{d / 2}}$, whose row and column indices we identify with the set $[N]^{d / 2}$ ordered lexicographically, having $Y_{(1 \cdots 1i)(1 \cdots 1j)} = X_{ij}$ for all $i, j \in [N]$ and satisfying the following properties:
    \begin{enumerate}
    \item $\bY \succeq \bm 0$.
    \item $Y_{\bs\bt}$ only depends on $\odd(\bs \circ \bt)$.
    \item $Y_{\bs\bt} = 1$ whenever $\odd(\bs \circ \bt) = \emptyset$.
    \end{enumerate}
    In this case, we say $\bY$ is a \emph{degree $d$ pseudomoment matrix over $\{ \pm 1\}^N$} (or, more properly, \emph{for the constraints $\{x_i^2 - 1 = 0: i \in [N]\}$}) which \emph{extends} $\bX$.
\end{definition}
\noindent
(In Appendix~\ref{app:sos-reductions} we present the standard reductions that allow us to restrict our attention to only the pseudomoments of degree exactly $d$ for the case of relaxing the problem $\M(\bW)$.)
In this paper, for the sake of brevity, we will simply refer to such $\bY$ as a \emph{degree $d$ pseudomoment matrix}, since we only study optimization over $\{ \pm 1\}^N$.

We will also only study degree 4 pseudomoment matrices in detail (though it will occasionally be instructive to refer to higher degree cases for comparison), so we briefly give a more concrete version of the above conditions for that case.
\begin{proposition}
    \label{prop:deg-4-pm-def}
    Let $\bY \in \RR^{N^2 \times N^2}$, with the row and column indices of $\bY$ identified with pairs $(ij) \in [N]^2$ ordered lexicographically.
    Then, $\bY$ is a degree 4 pseudomoment matrix if and only if the following conditions hold:
    \begin{enumerate}
    \item $\bY \succeq \bm 0$.
    \item $Y_{(ij)(kk)}$ does not depend on the index $k$.
    \item $Y_{(ii)(ii)} = 1$ for every $i \in [N]$.
    \item $Y_{(ij)(k\ell)}$ is invariant under permutations of the indices $i, j, k, \ell$.
    \end{enumerate}
\end{proposition}
\noindent
The intuitive meaning of these conditions is that $\bY$ contains \emph{pseudomoments} of degree 4 of a fictitious distribution over $\bx \in \{ \pm 1\}^N$, entry $Y_{(ij)(k\ell)}$ equaling the \emph{pseudoexpectation} of $x_ix_jx_kx_\ell$.
We represent the constraint $\bx \in \{ \pm 1\}^N$ as the system of polynomial constraints $x_i^2 = 1$ for $i \in [N]$, which constrains the pseudomoments per Conditions 2 and 3 of the definition.
The remaining constraints are properties that the moments of probability distributions in general must satisfy.
Conditions 1 through 4 taken together, however, still do not imply that a distribution over $\{ \pm 1\}^N$ exists whose moments equal the specified pseudomoments.

The matrix $\bX$, a minor of $\bY$, contains the degree 2 pseudomoments (which, as also discussed in greater detail in Appendix~\ref{app:sos-reductions}, are in this case already among the degree 4 pseudomoments).
We may then think of $\fE_4^N$ relaxing $\fC^N$ in the sense that $\fC^N$ is the set of degree 2 moment matrices of true distributions over $\{ \pm 1\}^N$, while $\fE_4^N$ is the set of degree 2 pseudomoment matrices that ``admit a consistent extension to degree 4.''

Note that the case $d > 2$ is fundamentally different from $d = 2$ in the important regard that, while $\fE_2^N$ is itself an affine slice of the psd cone (making it a so-called \emph{spectrahedron}), there does not appear to be a straightforward way to describe any $\fE_d^N$ with $d > 2$ in this way.
(On the other hand, it does not appear to be established that $\fE_d^N$ actually fails to be a spectrahedron for $d > 2$, though this is a natural conjecture.
Some similar results in simpler cases are presented in \cite{blekherman:12}.)
Rather, these sets are most directly described as projections of spectrahedra (so-called \emph{spectrahedral shadows}), the sets of all degree $d$ pseudomoment matrices, from a higher-dimensional space.
Thus, we should expect the generalized elliptopes to have a more subtle geometry than the classical elliptope, and our subject $\fE_4^N$ is the first of these subtler cases.

No more than the above is needed to understand our results, but for a more general and detailed presentation of the pseudomoment-and-pseudoexpectation optimization framework we direct the reader to \cite{barak:14:survey,lasserre:01,laurent:09}.

\paragraph{Convex geometry.}
We next recall some basic notions from the geometry of convex sets.
In what follows, let $K \subseteq \RR^d$ be a closed convex set.

\begin{definition}
    The \emph{dimension} of $K$ is the dimension of the affine hull of $K$, denoted $\dim(K)$.
\end{definition}

\begin{definition}
    A convex subset $F \subseteq K$ is a \emph{face of $K$} if whenever $\theta\bX + (1 - \theta)\bY \in F$ with $\theta \in (0, 1)$ and $\bX, \bY \in K$, then $\bX, \bY \in F$.
\end{definition}

\begin{definition}
    $\bX \in K$ is an \emph{extreme point of $K$} if $\{\bX\}$ is a face of $K$ (of dimension zero).
\end{definition}

\begin{definition}
    The intersection of all faces of $K$ containing $\bX \in K$ is the unique \emph{smallest face of $K$ containing $\bX$}, denoted $\mathsf{face}_K(\bX)$.
\end{definition}

\begin{definition}
    The \emph{perturbation of $\bX$ in $K$} is the subspace
    \begin{equation}
        \mathsf{pert}_K(\bX) \colonequals \left\{\bA \in \RR^d: \bX \pm t\bA \in K \text{ for all } t > 0 \text{ sufficiently small }\right\}.
    \end{equation}
\end{definition}
\noindent
The perturbation will come up naturally in our results, so we present the following useful fact giving its connection to the more intuitive objects from facial geometry.
\begin{proposition}
    Let $\bX \in K$.
    Then,
    \begin{equation}
        \mathsf{face}_K(\bX) = K \cap \left(\bX + \mathsf{pert}_K(\bX)\right).
    \end{equation}
    In particular, the affine hull of $\mathsf{face}_K(\bX)$ is $\bX + \mathsf{pert}_K(\bX)$, and therefore
    \begin{equation}
        \dim(\mathsf{face}_K(\bX)) = \dim(\mathsf{pert}_K(\bX))
    \end{equation}
    (in which there is a harmless reuse of notation between the dimension of a convex set and the dimension of a subspace).
\end{proposition}
\begin{proof}
    Let $G \colonequals K \cap (\bX + \mathsf{pert}_K(\bX))$.
    It is simple to check that $\bY \in G$ if and only if there exists $\bY^\prime \in K$ and $\theta \in (0, 1]$ with $\bX = \theta \bY + (1 - \theta)\bY^\prime$ (and if $\theta < 1$ then $\bY^\prime \in G$ as well).

    Then, if $F$ is any face of $K$ containing $\bX$, and $\bY \in G$, there exists $\bY^\prime \in K$ and $\theta \in (0, 1]$ such that $\bX = \theta \bY + (1 - \theta)\bY^\prime$.
    If $\theta = 1$, then $\bY = \bX \in F$.
    Otherwise, $\bY \in F$ by the definition of a face.
    Thus, in any case $\bY \in F$, so $G \subseteq F$.
    Since this holds for any face $F$ containing $\bX$, in fact $G \subseteq \mathsf{face}_{K}(\bX)$.

    It then suffices to show that $G$ is a face of $K$.
    Suppose $\bY \in G$, and $\bY_1, \bY_2 \in K$ and $\theta \in (0, 1)$ with $\bY = \theta\bY_1 + (1 - \theta)\bY_2$.
    Since $\bY \in G$, there exists $\bZ \in K$ and $\phi \in (0, 1]$ such that
    \begin{align}
      \bX
      &= \phi \bY + (1 - \phi)\bZ \nonumber \\
      &= \phi \big(\theta\bY_1 + (1 - \theta)\bY_2\big) + (1 - \phi)\bZ \nonumber \\
      &= \phi \theta \bY_1 + \phi(1 - \theta)\bY_2 + (1 - \phi)\bZ.
    \end{align}
    This is a convex combination of three points where the coefficients of $\bY_1$ and $\bY_2$ are strictly positive, so by the previous characterization we have $\bY_1, \bY_2 \in G$, completing the proof.
\end{proof}

\paragraph{Finite-dimensional frames.}
Finally, we review some definitions of special types of \emph{frames} in finite dimension, which are overcomplete collections of vectors with certain favorable geometric properties.
A more thorough introduction, in particular for the more typical applications of these definitions in signal processing and harmonic analysis, may be found in \cite{casazza:12}.
In what follows, let $\bv_1, \dots, \bv_N \in \RR^r$ be unit vectors and let $\bX \colonequals \Gram(\bv_1, \dots, \bv_N)$ (meaning that $X_{ij} = \la \bv_i, \bv_j\ra$ for $i, j \in [N]$).

\begin{definition}
    The vectors $\bv_i$ form a \emph{unit norm tight frame (UNTF)} if any of the following equivalent conditions hold:
    \begin{enumerate}
    \item $\sum_{i = 1}^N \bv_i\bv_i^\top = \frac{N}{r}\bm I_r$.
    \item The eigenvalues of $\bX$ all equal either zero or $\frac{N}{r}$.
    \item $\sum_{i = 1}^N \sum_{j = 1}^N \la \bv_i, \bv_j \ra^2 = \frac{N^2}{r}$.
    \end{enumerate}
\end{definition}
\noindent
(The equivalence of the final condition is elementary but less obvious; the quantity on its left-hand side is sometimes called the \emph{frame potential} \cite{benedetto:03}.)

\begin{definition}
    The $\bv_i$ form an \emph{equiangular tight frame (ETF)} if they form a UNTF, and there exists $\alpha \in [0, 1]$, called the \emph{coherence} of the ETF, such that whenever $i \neq j$ then $|X_{ij}| = \alpha$.
\end{definition}
\noindent
The following remarkable result shows that ETFs are extremal among UNTFs in the sense of \emph{worst-case coherence}.
Moreover, when an ETF exists, $\alpha$ is determined by $N$ and $r$.
\begin{proposition}[Welch Bound \cite{welch:74}]
    \label{prop:etf-welch-bound}
    If $\bv_1, \dots, \bv_N \in \RR^{r}$ with $\|\bv_i\|_2 = 1$, then
    \begin{equation}
        \max_{1 \leq i, j \leq N \atop i \neq j} |\la \bv_i, \bv_j \ra| \geq \sqrt{\frac{N - r}{r(N - 1)}},
    \end{equation}
    with equality if and only if $\bv_1, \dots, \bv_N$ form an ETF.
\end{proposition}
\noindent
ETFs usually arise from combinatorial constructions and should generally be understood as rigid and highly structured objects.
For instance, there remain many open problems about the pairs of dimensions $(N, r)$ for which ETFs do or do not exist.
More comprehensive references on these aspects of the theory of ETFs include \cite{tropp:07,casazza:08,fickus:15:tables}.

\subsection{Gramian Description of $\fE_4^N$}
Rather than directly working with the pseudomoment interpretation, we will study $\fE_4^N$ by pursuing an analogy with the following geometric description of the elliptope:
\begin{equation}
    \fE_2^N = \left\{\bX \in \RR^{N \times N}_{\sym}: \exists  \bv_1, \dots, \bv_N \in \RR^r \text{ such that } \|\bv_i\|_2 = 1 \text{ and } \bX = \Gram(\bv_1, \dots, \bv_N)\right\},
\end{equation}
Besides being used in many geometric arguments about the elliptope, which we will review later, in applications this description is central to the rounding procedures of \cite{goemans:95,nesterov:98} as well as the efficient rank-constrained approximations of \cite{burer:03}.
Our first result gives an analogous description of $\fE_4^N$ as Gram matrices of certain collections of unit vectors.
The following family of matrices plays an important role in this description.
\begin{definition}
    \label{def:blocksym}
    For $\bM \in \RR^{rN \times rN}$, we write $\bM_{[ij]}$ with $i, j \in [N]$ for the $r \times r$ block in position $(i, j)$ when $\bM$ is viewed as a block matrix.
    With this notation, let $\sB(N, r) \subset \RR^{rN \times rN}_{\sym}$ consist of matrices $\bM$ satisfying the following properties:
    \begin{enumerate}
    \item[1.] $\bM \succeq 0$.
    \item[2.] $\bM_{[ii]} = \bm I_r$ for all $i \in [N]$.
    \item[3.] $\bM_{[ij]} = \bM_{[ij]}^\top$ for all $i, j \in [N]$.
    \end{enumerate}
\end{definition}
\noindent
In terms of these matrices, $\fE_4^N$ admits the following description.
\begin{theorem}
    \label{thm:sos4}
    Let $\bv_1, \dots, \bv_N \in \RR^{r}$, let $\bX \colonequals \Gram(\bv_1, \dots, \bv_N) \in \RR^{N \times N}_{\sym}$, let $\bV \in \RR^{r \times N}$ have $\bv_1, \dots, \bv_N$ as its columns, and let $\bv \colonequals \vec(\bV) \in \RR^{rN}$ be the concatenation of the $\bv_i$.
    Then, $\bX \in \fE_4^N$ if and only if $\sum_{i = 1}^N\|\bv_i\|_2^2 = N$ and there exists $\bM \in \sB(N, r)$ such that $\bv^\top \bM \bv = N^2$.

    Moreover, if $\bX \in \fE_4^N$ and a degree 4 pseudomoment matrix $\bY \in \RR^{N^2 \times N^2}_{\sym}$ extends $\bX$, then there exists $\bM \in \sB(N, r)$ with $\bv^\top \bM \bv = N^2$ and
    \begin{align}
      \bY &= (\bm I_N \otimes \bV)^\top \bM (\bm I_N \otimes \bV), \text{ i.e.} \label{eq:M-Y-equiv-mx} \\
        Y_{(ij)(k\ell)} &= \bv_i^\top \bM_{[jk]} \bv_\ell \text{ for all } i, j, k, \ell \in [N]. \label{eq:M-Y-equiv}
    \end{align}
    Conversely, if $\sum_{i = 1}^N\|\bv_i\|_2^2 = N$ and $\bM \in \sB(N, r)$ with $\bv^\top \bM \bv = N^2$, then $\bY$ as defined by \eqref{eq:M-Y-equiv-mx} is a degree 4 pseudomoment matrix extending $\bX$.
\end{theorem}
\noindent
We think of $\bM$ as a \emph{witness} of the fact that $\bX \in \fE_4^N$, an alternative to the pseudomoment witness $\bY$ described in Proposition~\ref{prop:deg-4-pm-def}.
The second, more detailed part of Theorem~\ref{thm:sos4} gives one direction of the equivalence between these two types of witness; the other direction will be described in the course of the proof in Section~\ref{sec:pf:sos4}.

\subsection{Constraints on Pseudomoment Extensions}

Through Theorem~\ref{thm:sos4}, we will next connect the structure of degree 4 pseudomoment extensions of $\bX \in \fE_2^N$ and the local geometry of $\fE_2^N$ near $\bX$.
Theorem~\ref{thm:sos4} describes the membership of $\Gram(\bv_1, \dots, \bv_n)$ in $\fE_4^N$ in terms of a semidefinite program, whose variable is $\bM \in \sB(N, r)$.
Studying the dual of this semidefinite program and arguing through complementary slackness, we find that the optimal $\bM$ is highly constrained, as follows.
\begin{lemma}
    \label{lem:subspace-constraint}
    Let $\bv_1, \dots, \bv_N \in \SS^{r - 1}$ be a spanning set, let $\bV \in \RR^{r \times N}$ have the $\bv_i$ as its columns, let $\bv \colonequals \vec(\bV) \in \RR^{rN}$ be the concatenation of $\bv_1, \dots, \bv_N$, let $\bX \colonequals \Gram(\bv_1, \dots, \bv_N) \in \fE_2^N$, and let $\bM^\star \in \sB(N, r)$ be such that $\bv^\top \bM^\star \bv = N^2$.

    Then, all eigenvectors of $\bM^\star$ with nonzero eigenvalue belong to the subspace
    \begin{equation}
        V_{\sym} \colonequals \left\{ \vec(\bS\bV) : \bS \in \RR^{r \times r}_{\sym} \right\} \subset \RR^{rN}.
    \end{equation}
    Additionally, $\bv$ is an eigenvector of $\bM^\star$ with eigenvalue $N$, and all eigenvectors of $\bM^\star$ with nonzero eigenvalue that are orthogonal to $\bv$ belong to the subspace
    \begin{equation}
        V_{\sym}^\prime \colonequals \left\{ \vec(\bS\bV) : \bS \in \RR^{r \times r}_{\sym}, \bv_i^\top \bS \bv_i = 0 \text{ for } i \in [N] \right\} \subset \RR^{rN}.
    \end{equation}
\end{lemma}

We next apply \eqref{eq:M-Y-equiv-mx}, which shows how a spectral decomposition of $\bM^\star$ gives an expression for the associated pseudomoment matrix $\bY$ as a sum of $\rank(\bM^\star)$ (not necessarily orthogonal) rank one matrices, which are constrained by Lemma~\ref{lem:subspace-constraint}.
It turns out that these latter constraints are similar to those appearing in results of \cite{li:94,laurent:96} connecting the smallest face of $\fE_2^N$ containing $\bX$ to $\mathsf{span}(\{\bv_i\bv_i^\top\}_{i = 1}^N)$, which lets us describe the constraints on $\bY$ concisely in terms of the local geometry of $\fE_2^N$ near $\bX$.

\begin{theorem}
    \label{thm:sos-pataki}
    Suppose $\bX \in \fE_4^N$ and $\bY$ is a degree 4 pseudomoment matrix extending $\bX$.
    Then, $\bY \succeq \vec(\bX)\vec(\bX)^\top$, and all eigenvectors of $\bY - \vec(\bX)\vec(\bX)^\top$ with nonzero eigenvalue belong to the subspace $\vec(\mathsf{pert}_{\fE_2^N}(\bX))$.

    Consequently,
    \begin{align}
      \rank(\bY)
      &\leq \dim\left(\mathsf{pert}_{\fE_2^N}(\bX)\right) + 1 \label{eq:sos-pataki-face} \\
      &= \frac{\rank(\bX)(\rank(\bX) + 1)}{2} - \rank(\bX^{\odot 2}) + 1          \label{eq:sos-pataki} \\
      &\leq \frac{\rank(\bX)(\rank(\bX) + 1)}{2}. \label{eq:sos-pataki-weak}
    \end{align}
    In particular, if $\bX$ is an extreme point of $\fE_2^N$ and is extendable to a degree 4 pseudomoment matrix $\bY$, then $\rank(\bY) = \rank(\bX) = 1$, and $\bX = \bx\bx^\top$ and $\bY = (\bx \otimes \bx)(\bx \otimes \bx)^\top$ for some $\bx \in \{\pm 1\}^N$.
\end{theorem}
\begin{remark}
    The matrix $\bY - \vec(\bX)\vec(\bX)^\top$ is quite natural in the pseudomoment framework: entry $(ij)(k\ell)$ of this matrix contains the difference between the pseudoexpectation of $x_ix_jx_kx_\ell$ and the product of the pseudoexpectations of $x_ix_j$ and $x_kx_\ell$.
    It is natural to think of this quantity as the \emph{pseudocovariance} of $x_ix_j$ and $x_kx_\ell$, and it is then not surprising that the SOS constraints imply that the pseudocovariance matrix is psd.
    We are not aware, however, of previous results on SOS optimization that make direct use of the pseudocovariance matrix.
\end{remark}
\noindent
(As we will discuss in Section~\ref{sec:witness-constraints}, the equality \eqref{eq:sos-pataki} is a previously known result of \cite{li:94,laurent:96}.
Recall also that $\mathsf{pert}_{\fE_2^N}(\bX)$ as a subspace has the same dimension as $\mathsf{face}_{\fE_2^N}(\bX)$ as a convex set.)
The final claim gives a strong, albeit non-quantitative, suggestion that $\fE_4^N$ is a substantially tighter relaxation of $\fC^N$ than $\fE_2^N$: it implies that no ``spurious'' extreme points of $\fE_2^N$ that are not already extreme points of $\fC^N$ persist after constraining to $\fE_4^N$.

The bounds \eqref{eq:sos-pataki} and \eqref{eq:sos-pataki-weak} are similar in form to the Pataki bound on the rank of extreme points of feasible sets of general SDPs \cite{pataki:98}.
Because of the very large number of linear constraints in SDPs arising from SOS optimization, however, the Pataki bound is less effective in this setting.
In particular, the Pataki bound gives, at best,
\begin{equation}
    \rank(\bY) \leq (1 + o_{N \to \infty}(1))\sqrt{2m}
    \label{eq:pataki}
\end{equation}
for $\bY$ an extreme point of the set of degree 4 pseudomoment matrices, where $m$ is the number of linear constraints in the definition of an $N^2 \times N^2$ degree 4  pseudomoment matrix.
The degree 4 SOS constraints require that for each subset $\{i, j, k, \ell\} \subseteq [N]$, the permutation invariance of $Y_{(ij)(k\ell)}$ be enforced.
There are $\binom{N}{4} \sim \frac{1}{24}N^4$ such subsets and 24 ``copies'' whose equality must be enforced for each, of which it suffices to consider 12 since the constraint that the matrix $\bY$ be symmetric is not counted.
Thus $m = (1 - o_{N \to \infty}(1))\frac{11}{24}N^4$, whereby the right-hand side of \eqref{eq:pataki} is $(1 - o_{N \to \infty}(1))\sqrt{\frac{11}{12}}N^2$.

On the other hand, even with the weaker of our bounds \eqref{eq:sos-pataki-weak} and the naive further bound $\rank(\bX) \leq N$, we obtain the stronger claim that \emph{any} degree 4 pseudomoment matrix, not necessarily an extreme point, has rank at most $\binom{N + 1}{2} \sim \frac{1}{2}N^2$.
Additionally, the stronger inequality \eqref{eq:sos-pataki} is tight, achieved for instance by the degree 4 pseudomoment matrix $\bY$ where $Y_{(ij)(k\ell)} = 1$ if each index $i, j, k, \ell$ appears an even number of times and $Y_{(ij)(k\ell)} = 0$ otherwise (this matrix is the average of $(\bx \otimes \bx)(\bx \otimes \bx)^\top$ over all $\bx \in \{ \pm 1\}^N$), which extends $\bX = \bm I_N$.
Our bound would also give improved control of $\rank(\bY)$ in the case where $\bX$ is low-rank (say, when $\rank(\bX) \leq \delta N$ for small enough $\delta$).
The question of whether, in fact, the degree 2 pseudomoments of an optimal degree 4 pseudomoment matrix are typically low-rank in practical or random problems appears not to have been studied extensively and would be an interesting question for future work.
The result \eqref{eq:sos-pataki-face} casts some doubt on this natural conjecture, as it implies that there is less ``room'' to build a degree 4 pseudomoment matrix
extending a degree 2 pseudomoment matrix that lies on a low-dimensional face of $\fE_2^N$.

\subsection{Examples from Equiangular Tight Frames}

We next analyze the highly structured case of ETF Gram matrices, where the constraints of the previous section may guide the search for a degree 4 pseudomoment matrix extending a given degree 2 pseudomoment matrix.
The main reason that it is convenient to work with ETFs is that, when $\bX$ is the Gram matrix of an ETF, then $|X_{ij}|$ takes only two values, 1 when $i = j$ and some $\alpha \in [0, 1]$ otherwise.
Therefore, in particular, the matrix $\bX^{\odot 2}$ is very simple,
\begin{equation}
    \bX^{\odot 2} = (1 - \alpha^2)\bm I_N + \alpha^2 \one_N\one_N^\top.
\end{equation}
As we have seen in Theorem~\ref{thm:sos-pataki}, $\bX^{\odot 2}$ is intimately related to $\mathsf{pert}_{\fE_2^N}(\bX)$ and therefore to the possible degree 4 pseudomoment extensions of $\bX$.
In the case of ETFs, its simple structure makes it possible to compute an explicit (albeit naive) guess for a degree 4 pseudomoment extension, which rather surprisingly turns out to be correct.

By such reasoning, we obtain a complete characterization of membership in $\fE_4^N$ for ETF Gram matrices $\bX$, which is quite simple in that it depends only on the dimension and rank of $\bX$.
This result is as follows.

\begin{theorem}
    \label{thm:etf}
    Let $\bv_1, \dots, \bv_N \in \RR^r$ form an ETF, and let $\bX \colonequals \Gram(\bv_1, \dots, \bv_N)$.
    Then, $\bX \in \fE_4^N$ if and only if $N < \frac{r(r + 1)}{2}$ or $r = 1$.
    If $r = 1$, then $\bX = \bx\bx^\top$ for $\bx \in \{ \pm 1\}^N$, and a degree 4 pseudomoment matrix $\bY$ extending $\bX$ is given by $\bY = (\bx \otimes \bx)(\bx \otimes \bx)^\top$.
    If $r > 1$ and $N < \frac{r(r + 1)}{2}$, then, letting $\bP_{\vec(\mathsf{pert}_{\fE_2^N}(\bX))}$ be the orthogonal projection matrix to $\vec(\mathsf{pert}_{\fE_2^N}(\bX)) \subset \RR^{N^2}$, a degree 4 pseudomoment matrix $\bY$ extending $\bX$ is given by
    \begin{align}
        \bY &= \vec(\bX)\vec(\bX)^\top + \frac{N^2(1 - \frac{1}{r})}{\frac{r(r + 1)}{2} - N}\bP_{\vec(\mathsf{pert}_{\fE_2^N}(\bX))}, \text{ i.e.} \label{eq:etf-pm-mx} \\
        Y_{(ij)(k\ell)} &= \frac{\frac{r(r - 1)}{2}}{\frac{r(r + 1)}{2} - N}(X_{ij}X_{k\ell} + X_{ik}X_{j\ell} + X_{i\ell}X_{jk}) - \frac{r^2\left(1 - \frac{1}{N}\right)}{\frac{r(r + 1)}{2} - N}\sum_{m = 1}^N X_{im}X_{jm}X_{km}X_{\ell m}.
                          \label{eq:etf-pm-values}
    \end{align}
\end{theorem}
\noindent
As we will discuss further in the sequel, it is always the case that $N \leq \frac{r(r + 1)}{2}$ for an ETF (this is the \emph{Gerzon bound} \cite{lemmens:91}), and the \emph{maximal ETFs} with $N = \frac{r(r + 1)}{2}$ are notoriously elusive combinatorial objects; for instance, they are known to exist for only four values of $N$, and the question of their existence is open for infinitely many values of $N$ \cite{fickus:15:tables}.
Our result invokes another regard in which these ETFs are extremal, which was in fact present but perhaps unnoticed in existing results (in particular in an elegant proof of the Gerzon bound that we will present later): maximal ETF Gram matrices are the only ETF Gram matrices that are extreme points of $\fE_2^N$ (thus, by Theorem~\ref{thm:sos-pataki}, these Gram matrices cannot belong to $\fE_4^N$).

In our argument it will become clear that the case of ETFs (those non-maximal ones that do belong to $\fE_4^N$) is perhaps the simplest possible situation for degree 4 pseudomoments over the hypercube: as shown in \eqref{eq:etf-pm-mx}, the degree 4 pseudomoment matrix $\bY$ will have only two distinct positive eigenvalues, and will equal of the sum of the rank one matrix $\vec(\bX)\vec(\bX)^\top$, which contributes the ``naive'' pseudomoment value $X_{ij}X_{k\ell}$, with a constant multiple of the projection matrix onto the subspace $\vec(\mathsf{pert}_{\fE_2^N}(\bX))$, which contributes the remaining ``symmetrization'' term appearing in \eqref{eq:etf-pm-values}.

\subsection{Applications}

\paragraph{New inequalities in $\fE_4^N$.}

There are many results in combinatorial optimization enumerating linear inequalities satisfied by $\fC^N$ (see e.g.\ \cite{deza:09:book}).
The practical purpose of this pursuit is that such linear inequalities may be included in linear programming (LP) relaxations of $\fC^N$, which are typically more efficient than the SDP relaxations we work with here.
The putative convenience of SDP relaxations is that they do not require their user to know specifically which inequalities will be relevant for a given problem; the psd constraint captures many relevant inequalities at once.
For theoretical understanding, however, it is again important to know which specific inequalities over $\fC^N$ are satisfied at which degrees of SOS relaxation, since those inequalities may then be used as analytical tools.

Yet, to the best of our knowledge, very few inequalities over $\fC^N$ are known to be satisfied in $\fE_4^N$ but not $\fE_2^N$; indeed, it appears that the only infinite such family known before this work was the \emph{triangle inequalities},
\begin{equation}
    -s_is_jX_{ij} - s_js_kX_{jk} - s_is_kX_{ik} \leq 1 \text{ for } \bX \in \fE_4^N, \bs \in \{ \pm 1\}^N.
\end{equation}
Guided by the results from the previous section, we find a new family of similar but independent inequalities.
First, from the negative result of Theorem~\ref{thm:etf}, we obtain concrete examples of matrices $\bX \in \fE_2^N \setminus \fE_4^N$, namely the Gram matrices of ETFs with $N = \frac{r(r + 1)}{2}$.
As mentioned before, these are only known to exist for four specific dimensions, namely $r \in \{2, 3, 7, 23\}$.
By convex duality, there must exist certificates that these matrices do not belong to $\fE_4^N$, taking the form of linear inequalities that hold over $\fE_4^N$ but fail to hold for these matrices.
Indeed, for the smallest two examples $r \in \{2, 3\}$, a triangle inequality is a valid certificate of infeasibility.

For $r = 7$, on the other hand, the absolute value of the off-diagonal entries of the Gram matrix is $\alpha = \frac{1}{3}$, so the triangle inequalities are satisfied, and the certificates of infeasibility must be new inequalities which cannot be obtained as linear combinations of triangle inequalities.
We compute these certificates numerically and identify the constants that arise by hand to allow the certificates to be validated by symbolic computation (this amounts to checking that a certain $N^2 \times N^2$ matrix is psd, where in this case $N^2 = 28^2 = 784$).

For $r = 23$ the same appears to occur numerically and a similar argument shows that yet another independent family of inequalities must arise as the certificates of infeasibility, but the symbolic verification of such a certificate is a much larger problem which a naive software implementation does not solve in a reasonable time.
We thus only present the verified result for $r = 7$ here as a proof of concept, leaving both further computational verification of exact inequalities and further theoretical analysis of these certificates to future work.

\begin{theorem}
    \label{thm:max-etf-7-ineqs}
    Let $\bZ$ be the Gram matrix of an ETF of $28$ vectors in $\RR^7$.
    Then, for any $\bX \in \fE_4^N$ and any $\pi: [28] \to [N]$ injective,
    \begin{equation}
        \label{eq:max-etf-7-ineqs}
        \sum_{1 \leq i < j \leq 28} \sgn(Z_{ij})X_{\pi(i)\pi(j)} \leq 112,
    \end{equation}
    and this inequality cannot be obtained as a linear combination of the triangle inequalities
    \begin{equation}
        -s_is_jX_{ij} - s_js_kX_{jk} - s_is_kX_{ik} \leq 1 \text{ for } \bs \in \{ \pm 1 \}^N.
    \end{equation}
\end{theorem}
\noindent
As we will detail in the sequel, there is a general correspondence between ETFs and strongly regular graphs (SRGs) \cite{fickus:15}, under which the ETF of 28 vectors in $\RR^7$ (which is unique up to negating a subset of its vectors) corresponds to the \Schlafli\ graph, a remarkably symmetrical 16-regular graph on 27 vertices (see e.g.\ \cite{cameron:80,chudnovsky:05} for examples of its structure and significance in combinatorics).
We thus refer to these inequalities as \emph{\Schlafli\ inequalities}.

As a point of comparison, since $\bZ \in \fE_2^N$, the half-space parallel to that defined by \eqref{eq:max-etf-7-ineqs} most tightly bounding $\fE_2^N$ must have the right-hand side at least
\begin{equation}
    \sum_{1 \leq i < j \leq 28} \sgn(Z_{ij})Z_{ij} = \sum_{1 \leq i < j \leq 28} |Z_{ij}| = \frac{28(28 - 1)}{2} \cdot \frac{1}{3} = 126 > 112.
\end{equation}
Thus, the \Schlafli\ inequalities describe directions in the vector space of symmetric matrices along which $\fE_4^N$ is strictly ``narrower'' than $\fE_2^N$.

\paragraph{New proofs of complexity of parity inequalities.}

Lastly, our results on ETFs imply a special case of the following theorem of Laurent, which shows a lower bound for what degree of the SOS hierarchy is required to describe the cut polytope exactly.
(In the same work this bound was conjectured to be optimal, which was later proven in \cite{fawzi:16}.)

\begin{proposition}[Theorems 5 and 6 of \cite{laurent:03}]
    \label{prop:laurent}
    Let $N \geq 3$ be odd.
    Define
    \begin{equation}
        \bX^{(N)} \colonequals \left(1 + \frac{1}{N - 1}\right)\bm I_N - \frac{1}{N - 1}\one\one^\top.
    \end{equation}
    Then, $\bX^{(N)} \in \fE_{N - 1}^N \setminus \fC^N$, and the degree $N - 1$ pseudomoment matrix $\bZ^{(N)} \in \RR^{[N]^{(N - 1) / 2} \times [N]^{(N - 1) / 2}}$ whose entries are given by
    \begin{equation}
        Z^{(N)}_{\bs \bt} \colonequals (-1)^{|\odd(\bs \circ \bt)| / 2} \prod_{\substack{i \text{ odd} \\ 1 \leq i < |\odd(\bs \circ \bt)|}} \frac{i}{N - i}
    \end{equation}
    extends $\bX^{(N)}$.
    In particular, $\fE_{N - 1}^N \neq \fC^N$.
\end{proposition}
The ``parity inequality'' we refer to in the heading of this section is the fact that $\bX^{(N)} \notin \fC^N$, which follows from the fact that when $N$ is odd, then $\one_N^\top\bX\one_N \geq 1$ for all $\bX \in \fC^N$ (which in turn follows from the fact that when $\bX = \bx\bx^\top$ with $\bx \in \{\pm 1\}^N$, then this simply says $(\sum_{i = 1}^N x_i)^2 \geq 1$), while $\one_N^\top \bX^{(N)} \one_N = 0$.
Proposition~\ref{prop:laurent} then says that this parity inequality fails to hold over $\fE_{N - 1}^N$.

Observing that $\bX^{(N)}$ is the Gram matrix of an ETF, we are able to reproduce a weaker version of this result as a corollary of Theorem~\ref{thm:etf}, where $\fE_{N - 1}^N$ is replaced with $\fE_4^N$.

\begin{corollary}
    \label{cor:laurent}
    Let $N \geq 4$.
    Then, the matrix $\bY \in \RR^{N^2 \times N^2}$ defined by
    \begin{equation}
      Y_{(ij)(k\ell)}^{(N)} = \left\{\begin{array}{lcl} 1 & : & |\odd((ijk\ell))| = 0, \\ -\frac{1}{N - 1} & : & |\odd((ijk\ell))| = 2, \\\frac{3}{(N - 1)(N - 3)} & : & |\odd((ijk\ell))| = 4 \end{array}\right.
    \end{equation}
    is a degree 4 pseudomoment matrix extending $\bX^{(N)}$.
\end{corollary}
\noindent
Though this result is weaker than the full Proposition~\ref{prop:laurent}, the technique of its proof is far simpler: we do not rely on any of the theory of association schemes or hypergeometric series used in the argument of \cite{laurent:03}, and we obtain the correct value for the degree 4 pseudomoments, the key result that $\bY \succeq \bm 0$, and explicit descriptions of the eigenvalues and eigenvectors of $\bY$ (via the matrix formula \eqref{eq:etf-pm-mx}) from a single straightforward linear algebra calculation.
We believe it is likely that a generalization of the methods presented here can yield the full result of Laurent's theorem in a similarly simplified fashion.

\section{Notations}

Boldface uppercase letters ($\bX$) denote matrices, and boldface lowercase letters ($\bx$) denote vectors.
Entries of matrices and vectors are written without boldface and with subscripts ($X_{ij}$, $x_i$), except if the vector or matrix itself has a subscript, in which case parentheses are used ($(\bx_i)_j$, $(\bX_i)_{jk}$).

$\bm I_d$ denotes the $d \times d$ identity matrix, $\one_d$ denotes the all-ones vector of length $d$, and $\hat{\one}_d$ denotes its normalization to a unit vector, $\hat{\one}_d = \frac{1}{\sqrt{d}}\one_d$.
The subscripts from these notations will be omitted when the suitable dimension is clear from context.

$\SS^{r - 1} \subset \RR^r$ denotes the sphere of unit radius in $\RR^r$.
$\RR^{k \times k}$ (resp.\ $\RR^{k \times k}_{\sym}$, $\RR^{k \times k}_{\antisym}$) denotes the set of $k \times k$ matrices (resp.\ symmetric, antisymmetric $k \times k$ matrices) with real entries.
$\sO(k)$ denotes the group of $k \times k$ orthogonal matrices.

For a matrix $\bV$, $\row(\bV)$ denotes the span of its rows, $\ker(\bV)$ denotes its kernel as an operator (also the orthogonal complement of $\row(\bV)$), $\mathsf{rank}(\bV)$ denotes its rank, and $\mathsf{null}(\bV)$ denotes its nullity, the dimension of its kernel.
For a vector $\bv \in \RR^k$, $\diag(\bv) \in \RR^{k \times k}$ is the diagonal matrix with $\diag(\bv)_{ii} = v_i$.
For matrices $\bX, \bY$ of the same dimensions, $\bX \odot \bY$ denotes the Hadamard or entrywise product, and $\bX^{\odot k}$ denotes the matrix obtained by taking the $k$th power of each entry of $\bX$.
For any subspace $V \subseteq \RR^r$, $\bP_{V} \in \RR^{r \times r}$ denotes the orthogonal projection matrix to $V$.
For $\bv_1, \dots, \bv_N \in \RR^k$, $\Gram(\bv_1, \dots, \bv_N) \in \RR^{N \times N}$ denotes the Gram matrix $(\la \bv_i, \bv_j \ra)_{i, j = 1}^N$.

We will work extensively with block matrices, for which we propose a non-standard but helpful notation. When $\bX \in \RR^{ab \times cd}$ and we have declared $\bX$ to be a $b \times d$ matrix of blocks of size $a \times c$, then $\bX_{[ij]} \in \RR^{a \times c}$ denotes the block in position $(i, j) \in [b] \times [d]$.
Similarly, when $\bx \in \RR^{ab}$ and we have declared $\bx$ to be divided into blocks of length $a$, then $\bx_{[i]} \in \RR^a$ denotes the block in position $i \in [b]$.
When $c = a$, so that the blocks of $\bX$ are square, we define the \emph{partial transpose} operation $\bX^{\ptop}$ to produce the matrix where every $a \times a$ block of $\bX$ is transposed.\footnote{This notation, standard in the quantum information literature, is justified by the symbol $\ptop$ being ``half of'' the transpose symbol $\top$.}

\section{Gramian Description of $\fE_4^N$: Theorem~\ref{thm:sos4}}
\label{sec:pf:sos4}

In this section we give the proof of Theorem~\ref{thm:sos4}, followed by some ancillary results providing intuitive interpretations of its statement and suggesting some connections to the notion of \emph{separability} as studied in quantum information theory.
Recall that Theorem~\ref{thm:sos4} describes an equivalence between the pseudomoment witness $\bY \in \RR^{N^2 \times N^2}$ extending $\bX \in \fE_2^N$ and the Gram vector witness $\bM \in \sB(N, r)$.
The basic idea is that $\bM$ describes certain rotations by which the blocks of $\bY$ must be related because of the pseudomoment matrix constraints.
We will show first how to build $\bM$ from $\bY$, and then how to build $\bY$ from $\bM$.

\subsection{Proof of Theorem~\ref{thm:sos4}: Pseudomoment Witness to Gram Vector Witness}

Let $\bY \in \RR^{N^2 \times N^2}$ be a degree 4 pseudomoment matrix extending $\bX \in \RR^{N \times N}$, where for some $\bv_1, \dots, \bv_N \in \RR^r$, $\bX = \Gram(\bv_1, \dots, \bv_N)$.
Let $\bV \in \RR^{r \times N}$ have the $\bv_i$ as its columns, and let $\bv = \vec(\bV) \in \RR^{rN}$ be the concatenation of $\bv_1, \dots, \bv_N$.
We will then show that there exists $\bM \in \sB(N, r)$ with $\bv^\top \bM \bv = N^2$ and
\begin{equation}
    \bY = (\bm I_N \otimes \bV)^\top \bM (\bm I_N \otimes \bV).
\end{equation}
We first analyze the special case $r = \rank(\bX)$, then extend to the general case.

\paragraph{Case 1: $r = \rank(\bX)$.}
We build $\bM$ based on a suitable factorization of $\bY$.
Let $r^\prime \colonequals \rank(\bY) \geq r$, then there exists $\bA \in \RR^{r^\prime \times N^2}$ such that $\bY = \bA^\top\bA$.
Let us expand in blocks
\begin{equation}
    \bA = \left[\begin{array}{cccc} \bA_1 & \bA_2 & \cdots & \bA_N \end{array}\right],
\end{equation}
for $\bA_i \in \RR^{r^\prime \times N}$.
Since $\bA_1^\top \bA_1 = \bY_{[11]} = \bX = \bV^\top \bV$, there exists $\bZ \in \RR^{r^\prime \times r}$ such that $\bA_1 = \bZ\bV$ and $\bZ^\top \bZ = \bm I_r$.
By adding extra columns, we may extend $\bZ$ to an orthogonal matrix $\tilde{\bZ} \in \sO(r^\prime)$.
The factorization $\bY = \bA^\top \bA$ is unchanged by multiplying $\bA$ on the left by any element of $\sO(r^\prime)$.
By performing this transformation with $\tilde{\bZ}$, we may assume without loss of generality that $\bA$ is chosen such that
\begin{equation}
    \bA_1 = \left[\begin{array}{c} \bV \\ \bm 0 \end{array}\right] \hspace{-0.15cm}\begin{array}{l} \} \hspace{0.1cm} r \\ \} \hspace{0.1cm} r^\prime - r\end{array}
\end{equation}
where the numbers following the braces show the dimensionality of the matrix blocks.

Now, since $\bA_i^\top \bA_i = \bY_{[ii]} = \bX = \bA_1^\top \bA_1$ for every $i \in [N]$ (since, by the degree 4 pseudomoment conditions, $Y_{(ik)(i\ell)} = Y_{(ii)(k\ell)} = Y_{(11)(k\ell)} = X_{k\ell}$), there must exist $\bm I_{r^\prime} = \bQ_1, \dots, \bQ_N \in \sO(r^\prime)$ such that $\bA_i = \bQ_i\bA_1$.
Let us expand $\bQ_i$ in blocks,
\begin{equation}
    \bQ_i = \big[\underbrace{\bU_i}_{r} \underbrace{\tilde{\bU}_i}_{r^\prime - r} \big].
\end{equation}
We then have
\begin{equation}
    \bA_i = \bQ_i \bA_1 = \bU_i \bV.
\end{equation}
(The extra variable $\tilde{\bU}_i$ will not be used in the argument.)
Therefore, the blocks of $\bY$ are given by
\begin{equation}
    \label{eq:sos4-pf-Y-blocks}
    \bY_{[ij]} = \bA_i^\top \bA_j = \bV^\top \bU_i^\top \bU_j \bV.
\end{equation}
By the permutation symmetry of $\bY$, every such block is symmetric.
Since $\bV$ has full rank, $\bV \bV^\top$ is invertible, and therefore the matrix $(\bV\bV^\top)^{-1}\bV\bY_{[ij]}\bV^\top(\bV\bV^\top)^{-1} = \bU_i^\top \bU_j$ is also symmetric.

We now define $\bM$ blockwise by
\begin{equation}
    \bM_{[ij]} \colonequals \bU_i^\top \bU_j.
\end{equation}
Then $\bM \succeq \bm 0$ by construction, $\bM_{[ii]} = \bm I_r$ since this is the upper left block of $\bQ_i^\top \bQ_i = \bm I_{r^\prime}$, and $\bM_{[ij]}$ is symmetric by the preceding derivation.
Thus, $\bM \in \sB(N, r)$.
By \eqref{eq:sos4-pf-Y-blocks}, we also have
\begin{equation}
    \bY = (\bm I_N \otimes \bV)^\top \bM (\bm I_N \otimes \bV).
\end{equation}
It remains only to check that $\bv^\top \bM \bv = N^2$:
\begin{equation}
  \bv^\top \bM \bv = \sum_{i = 1}^N\sum_{j = 1}^N \bv_i^\top \bM_{[ij]} \bv_j = \sum_{i = 1}^N\sum_{j = 1}^N (\bV^\top \bU_i^\top \bU_j \bV)_{ij} = \sum_{i = 1}^N \sum_{j = 1}^N Y_{(ii)(jj)} = N^2,
\end{equation}
completing the proof of the first case.

\paragraph{Case 2: $r > \rank(\bX)$.}

We will reduce this case to the previous case.
Let $r_0 = \rank(\bX) < r$.
Fix Gram vectors $\bv_1, \dots, \bv_N \in \RR^{r_0}$ such that $\bX = \Gram(\bv_1, \dots, \bv_N)$, and, by the previous argument, choose $\bM \in \sB(N, r_0)$ having $\bv^\top \bM \bv = N^2$.

Suppose that $\bv_1^\prime, \dots, \bv_N^\prime \in \RR^{r}$ such that $\bX = \Gram(\bv_1^\prime, \dots, \bv_N^\prime)$.
Let $\bv^\prime$ be the concatenation of $\bv_1^\prime, \dots, \bv_N^\prime$.
Since the Gram matrices of $\bv_1, \dots, \bv_N$ and $\bv_1^\prime, \dots, \bv_N^\prime$ are equal, there must exist $\bZ \in \RR^{r \times r_0}$ with $\bZ\bv_i = \bv_i^\prime$ for each $i \in [N]$ and $\bZ^\top \bZ = \bm I_{r_0}$.
Define $\bM^\prime \in \RR^{r N \times r N}$ to have blocks
\begin{equation}
    \bM^\prime_{[ij]} \colonequals \left\{\begin{array}{lcr} \bZ \bM_{[ij]} \bZ^\top & : & i \neq j, \\ \bm I_{r} & : & i = j. \end{array}\right.
\end{equation}
Equivalently,
\begin{equation}
    \bM^\prime = (\bm I_N \otimes \bZ) \bM (\bm I_N \otimes \bZ)^\top + \bm I_N \otimes (\bm I_{r} - \bZ\bZ^\top).
\end{equation}
Since $\bZ\bZ^\top \preceq \bm I_{r}$ (the left-hand side is a projection matrix), $\bM^\prime \succeq \bm 0$, and by construction $\bM^\prime_{[ii]} = \bm I_{r}$ and $\bM^\prime_{[ij]}$ is symmetric.
Thus, $\bM^\prime \in \sB(N, r)$.

We also have
\begin{align}
  \bv^{\prime^\top} \bM^\prime \bv^\prime
  &= \sum_{i = 1}^N \|\bv_i^\prime\|_2^2 + \sum_{\substack{1 \leq i, j \leq N  \\ i \neq j}} \bv_i^{\prime^\top}\bM^\prime_{[ij]}\bv_j^\prime \nonumber \\
  &= N + \sum_{\substack{1 \leq i, j \leq N  \\ i \neq j}} \bv_i^{\top}\bM_{[ij]}\bv_j \nonumber \\
  &= N^2.
\end{align}
Lastly, we check the formula for the entries of $\bY$, now distinguishing the cases $i = j$ and $i \neq j$:
\begin{align}
  Y_{(ii)(k\ell)}
  &= X_{k\ell} = \la \bv^\prime_k, \bv^\prime_\ell \ra = \bv_k^{\prime^\top} \bM_{[ii]}^\prime \bv_\ell^\prime, \\
  Y_{(ij)(k\ell)}
  &= \bv_k^\top \bM_{[ij]} \bv_\ell \nonumber \\
  &= \bv_k^{\prime^\top} \bZ \bM_{[ij]} \bZ^\top \bv_\ell^\prime \nonumber \\
  &=\bv_k^{\prime^\top} \bM_{[ij]}^\prime \bv_\ell^\prime \text{   }(\text{for } i \neq j),
\end{align}
completing the proof.

\subsection{Proof of Theorem~\ref{thm:sos4}: Gram Vector Witness to Pseudomoment Witness}

Suppose that $\bX = \Gram(\bv_1, \dots, \bv_N) \in \RR^{N \times N}$ for some $\bv_i \in \RR^r$ having $\sum_{i = 1}^N \|\bv_i\|_2^2 = N$.
Let $\bv$ be the concatenation of $\bv_1, \dots, \bv_N$.
Suppose also that $\bM \in \sB(N, r)$ with $\bv^\top \bM \bv = N^2$.
We will show that $\bY \in \RR^{N^2 \times N^2}$ defined by
\begin{equation}
    Y_{(ij)(k\ell)} = \bv_i^\top \bM_{[jk]} \bv_\ell
\end{equation}
is a degree 4 pseudomoment matrix.
Recall that this requires the following properties to hold:
\begin{enumerate}
\item $\bY \succeq \bm 0$.
\item $Y_{(ij)(kk)}$ does not depend on the index $k$.
\item $Y_{(ii)(ii)} = 1$ for every $i \in [N]$.
\item $Y_{(ij)(k\ell)}$ is invariant under permutations of the indices $i, j, k, \ell$.
\end{enumerate}
(That the upper left $N \times N$ block of $\bY$ is $\bX$ follows from Property 4 and that $\bM_{[ii]} = \bm I_r$.)
We will obtain these one by one below.
This essentially just entails reversing the derivation of the previous part; however, verifying some of the properties of $\bY$ will require a more detailed understanding of the factorization of $\bM$ that we used.

The simplest is Property 1: since $\bM \succeq 0$, there exist some $\bU_1, \dots, \bU_N \in \RR^{r^\prime \times r}$ for some $r^\prime \geq 1$ such that $\bM_{[jk]} = \bU_j^\top \bU_k$.
Thus,
\begin{equation}
    Y_{(ij)(k\ell)} = \bv_i^\top \bU_j^\top \bU_k \bv_\ell = \la \bU_j\bv_i, \bU_k\bv_\ell \ra,
\end{equation}
so $\bY = \Gram(\bU_1\bv_1, \dots, \bU_N\bv_N) \succeq \bm 0$.

For Properties 2 and 3, we will first need some basic results on the spectrum of $\bM \in \sB(N, r)$.
The proofs of these facts are given in Appendix~\ref{app:BNr-structure}.
\begin{proposition}
    \label{prop:blocksym-structure}
    Let $\bM \in \sB(N, r)$. Then,
    \begin{enumerate}
    \item $\|\bM_{[ij]}\| \leq 1$ for all $i, j \in [N]$;
    \item $\|\bM\| \leq N$;
    \item if $\bM\bv = N\bv$, and $\bm 0 \neq \bv \in \RR^{rN}$ is the concatenation of $\bv_i \in \RR^r$, then the norms $\|\bv_i\|_2$ are all equal, and $\bM_{[ij]}\bv_j = \bv_i$ for all $i, j \in [N]$.
    \end{enumerate}
\end{proposition}
\noindent
From Claim 2 in the Proposition, since $\|\bv\|_2^2 = \Tr(\bX) = N$, then if $\bv^\top \bM \bv = N^2$ we must have $\bM \bv = N\bv$.
Therefore, by Claim 3, $\|\bv_i\|_2 = 1$ for each $i \in [N]$.
Also by Claim 3, we have
\begin{equation}
    Y_{(ij)(kk)} = \bv_i^\top \bM_{[jk]} \bv_k = \la \bv_i, \bv_j \ra.
\end{equation}
This gives Property 2, and taking $i = j = k$ gives Property 3 since $\|\bv_i\|_2 = 1$.

Property 4 is more subtle to establish.
First, for a moment treating $i, j, k, \ell$ as merely four distinct symbols, note that the symmetric group on $\{i, j, k, \ell\}$ is generated by the three transpositions $(ij)$, $(jk)$, and $(k\ell)$.
Therefore, to establish Property 4 it suffices to show the three equalities
\begin{equation}
    Y_{(ij)(k\ell)} = Y_{(ji)(k\ell)} = Y_{(ij)(\ell k)} = Y_{(ik)(j\ell)}
    \label{eq:sufficient-permutation-equalities}
\end{equation}
for all $i, j, k, \ell \in [N]$.
One equality follows directly from both $\bM_{[jk]}$ and $\bM$ being symmetric, whereby $\bM_{[jk]} = \bM_{[kj]}$:
\begin{equation}
    Y_{(ij)(k\ell)} = \bv_i^\top \bM_{[jk]} \bv_\ell = \bv_i^\top \bM_{[kj]} \bv_\ell = Y_{(ik)(j\ell)}.
\end{equation}
The others require a more detailed argument involving a factorization of $\bM \in \sB(N, r)$, as follows.
\begin{proposition}
    \label{prop:blocksym-factorization}
    Let $\bM \in \sB(N, r)$.
    Then, there exists $\bU \in \RR^{r^\prime \times rN}$ for some $r \leq r^\prime \leq rN$ such that $\bM = \bU^\top \bU$, where
    \begin{equation}
        \bU = \left[
            \begin{array}{cccc}
              \bS_1 & \bS_2 & \cdots & \bS_N \\
              \bR_1 & \bR_2 & \cdots & \bR_N
            \end{array}\right]
    \end{equation}
    for some $\bS_i \in \RR^{r \times r}_{\sym}$, $\bS_1 = \bm I_r$, $\bR_i \in \RR^{(r^\prime - r) \times r}$, $\bR_1 = \bm 0$, which satisfy the relations
    \begin{align}
      \bS_i^2 + \bR_i^\top \bR_i &= \bm I_r, \label{eq:s-r-diag} \\
      \bS_i \bS_j - \bS_j \bS_i + \bR_i^\top \bR_j - \bR_j^\top \bR_i &= \bm 0. \label{eq:s-r-commutator}
    \end{align}
\end{proposition}
\noindent
(The latter relations encode the conditions $\bM_{[ii]} = \bm I_r$ and $\bM_{[ij]}^\top = \bM_{[ij]}$, respectively.)
Combining Proposition~\ref{prop:blocksym-structure}'s Claim 3 and Proposition~\ref{prop:blocksym-factorization}, we find that
\begin{align}
  \bv_i &= \bM_{[i1]}\bv_1 = \bS_i\bv_1, \label{eq:i-i-1} \\
  \bv_1 &= \bM_{[1i]}\bv_i = \bS_i \bv_i. \label{eq:1-i-i}
\end{align}
We may therefore expand the entries of $\bY$ in terms of the matrices $\bS_i$ and $\bR_i$ and the vector $\bv_1$:
\begin{align}
  Y_{(ij)(k\ell)}
  &= \bv_i^\top \bM_{[jk]} \bv_\ell \nonumber \\
  &= \bv_1^\top \bS_i (\bS_j\bS_k + \bR_j^\top \bR_k)\bS_\ell \bv_1 \nonumber \\
  &= \bv_1^\top \bS_i\bS_j\bS_k\bS_\ell \bv_1 + \bv_1^\top \bS_i \bR_j^\top \bR_k \bS_\ell \bv_1.
\end{align}
To show the first two equalities of \eqref{eq:sufficient-permutation-equalities}, it then suffices to show that for any $i, j \in [N]$, we have
\begin{align}
  \bS_i\bS_j \bv_1 &\eqqu \bS_j\bS_i \bv_1, \label{eq:s-comm-req} \\
  \bR_i\bS_j \bv_1 &\eqqu \bR_j\bS_i \bv_1. \label{eq:r-comm-req}
\end{align}

Observe first that, by \eqref{eq:i-i-1} and \eqref{eq:1-i-i}, we have
\begin{equation}
    \bS_i^2 \bv_1 = \bv_1.
\end{equation}
Taking \eqref{eq:s-r-diag} as a quadratic form with $v_1$, we find
\begin{equation}
  1 = \|\bv_1\|_2^2 = \bv_1^\top \bS_i^2 \bv_1 + \|\bR_i\bv_1\|_2^2 = 1 + \|\bR_i\bv_1\|_2^2,
\end{equation}
hence $\bR_i \bv_1 = \bm 0$ for all $i \in [N]$.
Then, multiplying \eqref{eq:s-r-commutator} on the right by $\bv_1$ establishes \eqref{eq:s-comm-req}.

Next, taking \eqref{eq:s-r-diag} as a quadratic form with $\bv_i = \bS_i\bv_1$, we find
\begin{equation}
  1 = \|\bv_i\|_2^2 = \|\bS_i \bv_i\|_2^2 + \|\bR_i\bv_i\|_2^2 = 1 + \|\bR_i\bv_i\|_2^2,
\end{equation}
so $\bR_i\bS_i\bv_1 = \bR_i\bv_i = \bm 0$ for each $i \in [N]$ as well.
Also, evaluating \eqref{eq:s-r-diag} as a quadratic form with $\bv_j = \bS_j\bv_1$, we have
\begin{equation}
  1 = \|\bv_j\|_2^2 = \|\bS_i \bS_j \bv_1\|_2^2 + \|\bR_i \bS_j \bv_1\|_2^2.
\end{equation}
Taking \eqref{eq:s-r-diag} as a bilinear form with $\bS_i\bv_1$ and $\bS_j\bv_1$ and using the two preceding observations gives
\begin{align}
  0
  &= \bv_1^\top \bS_j(\bS_i \bS_j - \bS_j \bS_i + \bR_i^\top \bR_j - \bR_j^\top \bR_i)\bS_i\bv_1 \nonumber \\
  &= \|\bS_i\bS_j\bv_1\|_2^2 - 1 + \la \bR_i \bS_j \bv_1, \bR_j\bS_i \bv_1 \ra \nonumber \\
  &= -\|\bR_i \bS_j \bv_1\|_2^2 + \la \bR_i \bS_j \bv_1, \bR_j\bS_i \bv_1 \ra.
\end{align}
The same holds with indices $i$ and $j$ exchanged, so we find
\begin{equation}
    \la \bR_i \bS_j \bv_1, \bR_j\bS_i \bv_1 \ra = \|\bR_i \bS_j \bv_1\|_2^2 = \|\bR_j \bS_i \bv_1\|_2^2 = \|\bR_i \bS_j \bv_1\|_2\|\bR_j \bS_i \bv_1\|_2.
\end{equation}
Thus the Cauchy-Schwarz inequality holds tightly between the vectors $\bR_i \bS_j \bv_1$ and $\bR_j\bS_i \bv_1$, so $\bR_i \bS_j \bv_1 = \bR_j\bS_i \bv_1$, establishing \eqref{eq:r-comm-req} and completing the proof.

\subsection{Interpreting Theorem~\ref{thm:sos4} as a Relaxation}

Since $\fE_{4}^N$ may be seen as a relaxation of the cut polytope $\fC^N$, one expects that the description of $\fE_4^N$ in terms of an SDP over the matrices of $\sB(N, r)$, as stated in Theorem~\ref{thm:sos4}, should itself relax a description of $\fC^N$ in terms of a similar SDP with additional non-convex constraints.
In this section, we show that the most naive such description one might expect is in fact incorrect, and give the correct description, which highlights an interesting connection with ideas from quantum information theory.
The proofs of the results we give are deferred to Appendix~\ref{app:sos4-relaxations}.

Naively, by analogy with the fact that if $\bX \in \fE_2^N$ with $\rank(\bX) = 1$ then $\bX = \bx\bx^\top$ for $\bx \in \{\pm 1\}^N$, one might expect that constraining the rank of $\bM \in \sB(N, r)$ in Theorem~\ref{thm:sos4} to be as small as possible, namely to equal $r$, would give a description of $\fC^N$.
Unfortunately, as the following result shows, this only holds in one direction: if the Gram vector witness $\bM$ has rank $r$ then the associated $\bX \in \fC^N$, but there exist $\bX \in \fC^N$ with $\rank(\bX) = r$ whose membership in $\fE_4^N$ does not admit a Gram vector witness $\bM$ with $\rank(\bM) = r$.
\begin{proposition}
    \label{prop:blocksym-rank-constrained}
    Let $\bv_1, \dots, \bv_N \in \RR^{r}$, let $\bX = \Gram(\bv_1, \dots, \bv_N)$, and let $\bv \in \RR^{rN}$ be the concatenation of $\bv_1, \dots, \bv_N$.
    Then, if $\sum_{i = 1}^N \|\bv_i\|_2^2 = N$ and there exists $\bM \in \sB(N, r)$ with $\rank(\bM) = r$ and $\bv^\top \bM \bv = N^2$, then $\bX \in \fC^N$.
    On the other hand, if $N \notin \{1, 2\}$ and $N$ is not divisible by 4, then $\bm I_N \in \fC^N$ with $\bm I_N = \Gram(\be_1, \dots, \be_N)$, but letting $\bv$ be the concatenation of $\be_1, \dots, \be_N$, there does not exist $\bM \in \sB(N, N)$ with $\bv^\top \bM \bv = N^2$ and $\rank(\bM) = N$.
\end{proposition}
\noindent
(The unusual condition on the negative result that $N$ be odd is probably superfluous if one searches for counterexamples other than the identity; the question is related to the relationship between the rank of a matrix in $\fC^N$ and the minimum number of cut matrices to whose convex hull it belongs, which, as is discussed in the proof, is known to be subtle.)

The correct way to ``repair'' this first attempt is quite surprising: the key condition on the Gram vector witness $\bM \in \sB(N, r)$ that is equivalent to $\bX \in \fC^N$ is not minimal rank, but \emph{separability}, a notion studied mainly in quantum information theory.
Related ideas will play a role in our derivation of constraints on the Gram vector witness $\bM$ defined in Theorem~\ref{thm:sos4}, but the full extent of this connection is still unclear and is an intriguing subject for future work.

\begin{definition}
    \label{def:blocksym-sep}
    A matrix $\bM \in \RR^{rN \times rN}$ with $\Tr(\bM) = 1$ is \emph{separable} if there exist $\ba_1, \dots, \ba_m \in \RR^N$ with $\|\ba_i\|_2 = 1$, $\bb_1, \dots, \bb_m \in \RR^r$ with $\|\bb_i\|_2 = 1$, and $\rho_1, \dots, \rho_m \geq 0$ with $\sum_i \rho_i = 1$ such that
    \begin{equation}
        \bM = \sum_{i = 1}^m \rho_i (\ba_i \otimes \bb_i)(\ba_i \otimes \bb_i)^\top.
    \end{equation}
    If it is not possible to write $\bM$ in this way, $\bM$ is \emph{entangled}.
    (More properly, $\bM$ is the \emph{density matrix} representing, with respect to a particular choice of basis, a \emph{bipartite quantum state}, and it is the state that is entangled or separable.)
    We write $\sB_{\sep}(N, r) \subseteq \sB(N, r)$ for the matrices $\bM \in \sB(N, r)$ such that $\frac{1}{rN}\bM$ is separable.
\end{definition}

\begin{proposition}
    \label{prop:blocksym-separable}
    Let $\bv_1, \dots, \bv_N \in \RR^{r}$, let $\bX = \Gram(\bv_1, \dots, \bv_N)$, and let $\bv \in \RR^{rN}$ be the concatenation of $\bv_1, \dots, \bv_N$.
    Then, $\bX \in \fC^N$ if and only if $\sum_{i = 1}^N \|\bv_i\|_2^2 = N$ and there exists $\bM \in \sB_{\sep}(N, r)$ such that $\bv^\top \bM \bv = N^2$.
\end{proposition}

By corollary, if $\bX \in \fE_4^N \setminus \fC^N$, then any Gram vector witness $\bM$ (suitably scaled) must be the density matrix of an entangled state which, by the definition of $\sB(N, r)$, has the \emph{positive partial transpose (PPT)} property that its partial transpose remains psd.
If the partial transpose of a density matrix of a state fails to be psd, it follows that the state is entangled, but the converse does not hold in general \cite{peres:96,horedecki:96}.
The structure of states for which this test does not prove entanglement but which are nonetheless entangled has received considerable attention in the quantum information literature (see e.g.\ \cite{lewenstein:00,szarek:06,leinaas:10,chen:12}, as well as \cite{jaeger:07:book,bengtsson:17:book,aubrun:17:book} for more general discussion).
It is therefore striking that these objects are, per our results, rather commonplace in SOS optimization---for every hypercube optimization problem for which degree 4 SOS is not tight (i.e.\ for which the optimizer $\bX^\star \in \fE_4^N \setminus \fC^N$), there is an underlying entangled PPT state that may be recovered from $\bX^\star$.

\section{Constraints on Witnesses: Lemma~\ref{lem:subspace-constraint} and Theorem~\ref{thm:sos-pataki}}
\label{sec:witness-constraints}

\subsection{Proof of Lemma~\ref{lem:subspace-constraint}}

Suppose that $\bX = \Gram(\bv_1, \dots, \bv_N) \in \fE_4^N$ for some $\bv_1, \dots, \bv_N \in \SS^{r - 1}$, $\bv$ is the concatenation of the $\bv_i$, and $\bv^\top \bM^\star \bv = N^2$ for some $\bM^\star \in \sB(N, r)$.
Then, $\bM^\star$ is an optimizer for the following SDP, described by Theorem~\ref{thm:sos4}:
\begin{equation}
    \label{eq:gram-sdp-primal}
    \GramSDP(\bv_1, \dots, \bv_N) \colonequals \left\{\begin{array}{ll}
      \text{maximize} & \la \bv\bv^\top, \bM \ra \\
      \text{subject to} & \bM \succeq 0, \\
                      & \bM_{[ii]} = \bm I_r, \\
                      & \bM_{[ij]} = \bM_{[ij]}^\top \text{ for } i \neq j.
    \end{array}\right\}.
\end{equation}
We next apply basic convex optimization results to this SDP.
Background on these general facts may be found in \cite{ben:01,boyd:04}.
First, we obtain the dual SDP
\begin{equation}
    \label{eq:gram-sdp-dual}
    \GramSDP^{*}(\bv_1, \dots, \bv_N) \colonequals \left\{\begin{array}{ll}
      \text{minimize} & \Tr(\bD) \\
      \text{subject to} & \bD \succeq \bv\bv^\top, \\
                      & \bD_{[ij]} = -\bD_{[ij]}^\top \text{ for } i \neq j.
    \end{array}\right\}.
\end{equation}
It is simple to verify that the Slater condition holds, implying strong duality between the SDPs \eqref{eq:gram-sdp-primal} and \eqref{eq:gram-sdp-dual}, whereby $\GramSDP(\bv_1, \dots, \bv_N) = \GramSDP^*(\bv_1, \dots, \bv_N) = N^2$.
If $\bM^\star$ and $\bD^\star$ are primal and dual variables achieving the optimal values of $\GramSDP$ and $\GramSDP^*$ respectively, then \emph{complementary slackness} must hold between them, $\bM^\star(\bD^\star - \bv\bv^\top) = \bm 0$.

The key to Lemma~\ref{lem:subspace-constraint} is that, while constructing $\bM^\star$ achieving a value of $N^2$ in $\GramSDP$ from $\bv_1, \dots, \bv_N$ (when their Gram matrix belongs to $\fE_4^N$) is difficult (indeed, by the more detailed part of Theorem~\ref{thm:sos4} it is equivalent to constructing the degree 4 pseudomoments themselves), constructing $\bD^\star$ achieving a value of $N^2$ in $\GramSDP^*$ turns out to be straightforward.

The construction uses the \emph{partial transpose} operation, $\bA \mapsto \bA^{\ptop}$ mapping $\RR^{rN \times rN} \to \RR^{rN \times rN}$, which transposes every $r \times r$ block of an $rN \times rN$ matrix:
\begin{equation}
    \bA^\ptop \colonequals \left[ \bA_{[ij]}^\top \right]_{i, j = 1}^N \text{ for } \bA \in \RR^{rN \times rN}.
\end{equation}
We then define $\bD^\star$ as
\begin{equation}
    \bD^\star \colonequals \bv\bv^\top - (\bv\bv^\top)^{\ptop} + \bm I_N \otimes (\bV\bV^\top).
\end{equation}
We have $\Tr(\bD^\star) = \Tr(\bm I_N \otimes (\bV\bV^\top)) = N^2$, and for $i \neq j$, $\bD_{[ij]} = \bv_i\bv_j^\top - \bv_j\bv_i^\top$, which is antisymmetric as required.
The final feasibility condition $\bD^\star \succeq \bv\bv^\top$ is equivalent to  $(\bv\bv^\top)^\ptop \preceq \bm I_N \otimes (\bV\bV^\top)$, which is simple to check directly: let $\bx \in \RR^{rN}$ be the concatenation of $\bx_1, \dots, \bx_N \in \RR^r$, then
\begin{align*}
  \bx^\top (\bv\bv^\top)^\ptop\bx
  &= \sum_{i = 1}^N \sum_{j = 1}^N \la \bx_i, \bv_j \ra \la \bx_j, \bv_i \ra \\
  &\leq \sum_{i = 1}^N \sum_{j = 1}^N \la \bx_i, \bv_j \ra^2 \\
  &= \sum_{i = 1}^N \bx_i^\top (\bV\bV^\top) \bx_i \\
  &= \bx^\top (\bm I_N \otimes (\bV\bV^\top)) \bx.
\end{align*}
Thus, by complementary slackness an optimizer $\bM^\star$ in $\GramSDP$ must have positive eigenvectors in $\ker(\bD^\star - \bv\bv^\top) = \ker(\bm I_N \otimes (\bV\bV^\top) - (\bv\bv^\top)^\ptop)$.

Remarkably, this subspace may be calculated exactly to produce the result of Lemma~\ref{lem:subspace-constraint}.
This calculation hinges on some elementary but perhaps not widely known facts of linear algebra that originate in applications to quantum information theory.
The first is the following, a rewriting of the singular value decomposition.

\begin{proposition}[Schmidt Decomposition, Section 2.2.2 of \cite{aubrun:17:book}]
    \label{prop:schmidt-decomp}
    Let $r \leq N$, $\bV \in \RR^{r \times N}$ having singular value decomposition $\bV = \sum_{i = 1}^r \sigma_i \by_i\bz_i^\top$, where the $\by_i \in \RR^r$ and $\bz_i \in \RR^N$ each form orthonormal sets and $\sigma_i \geq 0$.
    Then,
    \begin{equation}
        \vec(\bV) = \sum_{i = 1}^r \sigma_i \bz_i \otimes \by_i.
    \end{equation}
\end{proposition}
\noindent
Note that in our case, $\bv = \vec(\bV)$.

This representation makes it convenient to work with the partial transpose; in particular, using the Schmidt decomposition, it is possible to diagonalize the partial transpose of a rank one matrix explicitly, as follows.
(This result appears to be folkloric in the quantum information literature; the references we give are unlikely to be the earliest.)
\begin{proposition}[Lemma III.3 of \cite{hildebrand:07}; Lemma 1 of \cite{johnston:18}]
    \label{prop:rank-one-pt}
    Let $\bV \in \RR^{r \times N}$ with $r \leq N$ and $\bV = \sum_{i = 1}^r \sigma_i \by_i\bz_i^\top$ where the $\by_i \in \RR^r$ and $\bz_i \in \RR^N$ each form orthonormal sets and $\sigma_i \geq 0$.
    Let $\bv = \vec(\bV)$.
    Then,
    \begin{equation}
        (\bv\bv^\top)^{\ptop} = \sum_{i = 1}^r \sigma_i^2 \bd_i\bd_i^\top + \sum_{1 \leq i < j \leq r} \sigma_i\sigma_j\bs_{ij}\bs_{ij}^\top -
        \sum_{1 \leq i < j \leq r} \sigma_i\sigma_j\ba_{ij}\ba_{ij}^\top
        \label{eq:pt-rank-one-spectral-decomp}
    \end{equation}
    where
    \begin{align}
      \bd_i &= \bz_i \otimes \by_i, \\
      \bs_{ij} &= \frac{1}{\sqrt{2}}\left(\bz_i \otimes \by_j + \bz_j \otimes \by_i\right), \\
      \ba_{ij} &= \frac{1}{\sqrt{2}}\left(\bz_i \otimes \by_j - \bz_j \otimes \by_i\right).
    \end{align}
    The $r^2$ vectors $\bd_i, \bs_{ij}, \ba_{ij}$ moreover have unit norm and are mutually orthogonal, so \eqref{eq:pt-rank-one-spectral-decomp} is a spectral decomposition (up to the removal of any terms whose coefficient is zero if $\bV$ is not full rank).
\end{proposition}

With this, it is straightforward to compute the subspace we are interested in, since $\bm I_N \otimes (\bV\bV^\top)$ can also be diagonalized explicitly in a basis of the same Kronecker products $\bz_i \otimes \by_j$.
\begin{proposition}
    \label{prop:v-sym}
    Let $\bV \in \RR^{r \times N}$ with $r \leq N$ have full rank, and let $\bv = \vec(\bV)$.
    Then,
    \begin{equation}
        \bm I_N \otimes (\bV\bV^\top) \succeq (\bv\bv^\top)^{\ptop}.
    \end{equation}
    The subspace on which this inequality is tight is given by
    \begin{equation}
        \ker\bigg(\bm I_N \otimes (\bV\bV^\top) - (\bv\bv^\top)^{\ptop}\bigg) = \left\{ \vec(\bS\bV) : \bS \in \RR^{r \times r}_{\sym}\right\} \equalscolon V_{\sym}.
    \end{equation}
    Letting $\bV = \sum_{i = 1}^r \sigma_i\by_i\bz_i^\top$ for $\by_i \in \RR^r$ an orthonormal basis, $\bz_i \in \RR^N$ an orthonormal set, and $\sigma_i > 0$ be the singular decomposition, an orthonormal basis for $V_{\sym}$ is given by the $\frac{r(r + 1)}{2}$ vectors
    \begin{align}
      \bz_i \otimes \by_i &\text{ for } 1 \leq i \leq N, \\
      \frac{1}{\sqrt{\sigma_i^2 + \sigma_j^2}}\left(\sigma_i \bz_i \otimes \by_j + \sigma_j \bz_j \otimes \by_i\right) &\text{ for } 1 \leq i < j \leq N.
    \end{align}
\end{proposition}
\noindent
We provide proofs of the preceding three Propositions in Appendix~\ref{app:partial-transpose}, giving more details on Proposition~\ref{prop:v-sym} since it is the only one of these results that appears to be original.

From the previous discussion of complementary slackness and Proposition~\ref{prop:v-sym}, we find that all eigenvectors with positive eigenvalue of $\bM$ must belong to $V_{\sym}$.
To obtain from this the statement of Lemma~\ref{lem:subspace-constraint}, first note that if $\bv^\top \bM \bv = N^2$ then by Proposition~\ref{prop:blocksym-structure} $\bM\bv = N\bv$, so $\bM = \bv\bv^\top + \bM^\prime$ for some $\bM^\prime \succeq \bm 0$.
Suppose that $\bw \in \RR^{rN}$ is an eigenvector of $\bM^\prime$ with eigenvalue $\lambda > 0$.
Then, $\bw \in V_{\sym}$ by the above reasoning, so $\bw = \vec(\bS\bV)$ for some $\bS \in \RR^{r \times r}_{\sym}$.
Also,
\begin{equation}
    \bm I_r = \bM_{[ii]} \succeq (\bv\bv^\top + \lambda\bw\bw^\top)_{[ii]} = \bv_i\bv_i^\top + \lambda \bS\bv_i\bv_i^\top \bS,
\end{equation}
and taking this as a quadratic form with $\bv_i$ shows that $\bv_i\bS \bv_i = 0$.
Since this holds for each $i \in [N]$, we obtain the conclusion of Lemma~\ref{lem:subspace-constraint}, that
\begin{equation}
    \bw \in V_{\sym}^\prime \colonequals \left\{ \vec(\bS\bV): \bS \in \RR^{r \times r}_{\sym}, \bv_i^\top \bS \bv_i = 0 \text{ for } i \in [N]\right\}.
\end{equation}

\subsection{Proof of Theorem~\ref{thm:sos-pataki}}

Suppose $\bX \in \fE_4^N$ with $\bX = \Gram(\bv_1, \dots, \bv_N)$, and $\bv_i \in \RR^r$ with $r = \rank(\bX)$.
Then if $\bV \in \RR^{r \times N}$ has the $\bv_i$ as its columns, $\bV$ is full-rank.
If $\bY \in \RR^{N^2 \times N^2}$ is any degree 4 pseudomoment matrix extending $\bX$, then there is $\bM \in \sB(N, r)$ with $\bv^\top \bM \bv = N^2$.
Suppose $r^\prime = \rank(\bM)$, then let us write the spectral decomposition $\bM = \bv\bv^\top + \sum_{m = 1}^{r^\prime - 1} \lambda_m\bw_m\bw_m^\top$ for some $\lambda_m > 0$.

By Lemma~\ref{lem:subspace-constraint}, $\bw_m \in V_{\sym}^\prime$.
Therefore, $\bw_m = \vec(\bS_m\bV)$ for some $\bS_m \in \RR^{r \times r}_{\sym}$ with $\bv_i^\top \bS_m \bv_i = 0$ for all $i \in [N], m \in [r^\prime - 1]$.
By \eqref{eq:M-Y-equiv-mx} from Theorem~\ref{thm:sos4}, we may therefore expand
\begin{align}
  \bY &= \tilde{\bv}\tilde{\bv}^\top + \sum_{m = 1}^{r^\prime - 1}\lambda_m\tilde{\bw}_m\tilde{\bw}_m^\top, \\
  (\tilde{\bv})_{(ij)} &= \la \bv_i, \bv_j \ra, \\
  (\tilde{\bw}_m)_{(ij)} &= \la \bS_m\bv_i, \bv_j\ra.
\end{align}
Thus, we simply have $\tilde{\bv} = \vec(\bX)$ and $\tilde{\bw}_m = \vec(\bV^\top \bS_m \bV)$.

The statement of Theorem~\ref{thm:sos-pataki} comes from combining this with the following previous result about the facial geometry of $\fE_2^N$.
\begin{proposition}[Theorem 1(a) of \cite{li:94}]
    \label{prop:li-tam}
    Let $\bX = \Gram(\bv_1, \dots, \bv_N) \in \fE_2^N$ for $\bv_1, \dots, \bv_N \in \SS^{r-1}$ having $\rank(\bX) = r$, and let $\bV \in \RR^{r \times N}$ have the $\bv_i$ as its columns, so that $\bX = \bV^\top \bV$.
    Then,
    \begin{align}
      \mathsf{pert}_{\fE_2^N}(\bX)
      &= \left\{ \bV^\top \bS \bV: \bS \in \RR^{r \times r}_{\sym} \right\} \cap \{ \bA \in \RR^{N \times N}: \diag(\bA) = \bm 0 \} \\
      &= \left\{ \bV^\top \bS \bV: \bS \in \RR^{r \times r}_{\sym}, \bv_i^\top \bS \bv_i = 0 \text{ for } i \in [N] \right\}.
    \end{align}
\end{proposition}
\noindent
Therefore, continuing the reasoning above, we find that for each $m \in [r^\prime - 1]$, $\tilde{\bw}_m = \vec(\bA_m)$ for some $\bA_m \in \mathsf{pert}_{\fE_2^N}(\bX)$.
Hence, every eigenvector of $\bY - \vec(\bX)\vec(\bX)^\top$ having nonzero eigenvalue must lie in $\vec(\mathsf{pert}_{\fE_2^N}(\bX))$, establishing the first part of the result.

The second part of the result controls $\rank(\bY) \leq r^\prime$.
By the first part of the result,
\begin{equation}
    r^\prime \leq \dim\left(\mathsf{pert}_{\fE_2^N}(\bX)\right) + 1,
\end{equation}
so it suffices to compute the right-hand side.
Since $\bV$ is full-rank, the map $\bS \mapsto \bV^\top \bS \bV$ is injective, so this may be computed as
\begin{equation}
    \dim\left(\mathsf{pert}_{\fE_2^N}(\bX)\right) = \dim\left(\mathsf{span}\left(\{\bv_i\bv_i^\top\}_{i = 1}^N\right)^\perp\right) = \frac{r(r + 1)}{2} - \dim\left(\mathsf{span}\left(\{\bv_i\bv_i^\top\}_{i = 1}^N\right)\right).
\end{equation}
Since $\Gram(\bv_1\bv_1^\top, \dots, \bv_N\bv_N^\top) = \bX^{\odot 2}$, we equivalently have
\begin{equation}
    \dim\left(\mathsf{pert}_{\fE_2^N}(\bX)\right) = \frac{r(r + 1)}{2} - \rank(\bX^{\odot 2}),
\end{equation}
a previously known corollary of Proposition~\ref{prop:li-tam} used in \cite{li:94,laurent:96}.

The final part of the result concerns the special case where $\bX \in \fE_2^N$ is an extreme point, whereby $\dim(\mathsf{pert}_{\fE_2^N}(\bX)) = 0$.
Then, if $\bY$ is a degree 4 pseudomoment matrix extending $\bX$ we have $\rank(\bY) = r^\prime = 1$, so $\rank(\bX) = 1$ as well since $\bX$ is a minor of $\bY$.
Since $\bX \in \fE_2^N$, in fact $\bX = \bx\bx^\top$ for some $\bx \in \{ \pm 1\}^N$, and it is simple to check that the only possible degree 4 extension of rank one is then $\bY = (\bx \otimes \bx)(\bx \otimes \bx)^\top$.

\section{Examples from Equiangular Tight Frames: Theorem~\ref{thm:etf}}
\label{sec:etfs}

Before giving the proofs of our results on ETFs, we first point out a general convenience of working with the Gram matrices of UNTFs through Theorem~\ref{thm:sos4}.
Suppose $\bX = \Gram(\bv_1, \dots, \bv_N)$ and the $\bv_i \in \RR^r$ form a UNTF.
Let $\bV \in \RR^{r \times N}$ have the $\bv_i$ as its columns.
Suppose also that $\bM \in \sB(N, r)$ with $\bv^\top \bM \bv = N^2$.
Lemma~\ref{lem:subspace-constraint} then ensures that the eigenvectors of $\bM$ with positive eigenvalue lie in the subspace of vectors of the form $\vec(\bS\bV)$ for $\bS \in \RR^{r \times r}_{\sym}$.
In the case where $\bv_1, \dots, \bv_N$ form a UNTF, we show that in fact this mapping is, up to scaling, an isometry.

\begin{definition}
    For $\bV \in \RR^{r \times N}$, let us write $\sV_{\bV}: \RR^{r \times r}_{\sym} \to \RR^{rN}$ for the map $\sV_{\bV}(\bS) = \sqrt{\frac{r}{N}}\vec(\bS\bV)$.
    When the matrix $\bV$ is clear from context, we will drop the subscript $\bV$.
\end{definition}

\begin{proposition}
    \label{prop:lin-isom}
    Let $\bv_1, \dots, \bv_N \in \SS^{r - 1}$ form a UNTF and let $\bV \in \RR^{r \times N}$ have the $\bv_i$ as its columns.
    Then, the mapping $\sV = \sV_{\bV}$ is a linear isometry between $\RR^{r \times r}_{\sym}$ and $\{\vec(\bS\bV): \bS \in \RR^{r \times r}_{\sym}\} \subset \RR^{rN}$, if $\RR^{r \times r}_{\sym}$ is endowed with the Frobenius inner product $\la \bS, \bS^\prime \ra = \Tr(\bS\bS^\prime)$.
\end{proposition}
\begin{proof}
    To check that inner products are preserved, we compute:
    \begin{align}
      \la \mathcal{V}(\bS), \mathcal{V}(\bS^\prime) \ra
      &= \frac{r}{N}\sum_{i = 1}^N \la \bS \bv_i, \bS^\prime \bv_i \ra \nonumber \\
      &= \frac{r}{N}\sum_{i = 1}^N \bv_i^\top \bS \bS^\prime \bv_i \nonumber \\
      &= \Tr\left(\bS \bS^\prime \left(\frac{r}{N}\bV\bV^\top \right)\right) \nonumber \\
      &= \la \bS, \bS^\prime \ra.
    \end{align}
    Clearly $\sV$ is linear, injectivity follows from the $\bv_i$ forming a spanning set, and surjectivity follows from the definition of the target space.
\end{proof}

Similarly, Theorem~\ref{thm:sos4} shows that $\bY$ can be produced from $\bM$ by conjugating as
\begin{equation}
    \bY = (\bm I_N \otimes \bV)^\top \bM (\bm I_N \otimes \bV) = \frac{N}{r}\left(\bm I_N \otimes \sqrt{\frac{r}{N}}\bV\right)^\top \bM \left(\bm I_N \otimes \sqrt{\frac{r}{N}}\bV\right),
\end{equation}
where in the latter expression the matrix $\bm I_N \otimes \sqrt{\frac{r}{N}}\bV$ has orthonormal rows, so $\bY$ is also merely a scaled and rotated copy of $\bM$, embedded in a higher-dimensional space.
In particular, the spectrum of $\bY$ is the spectrum of $\bM$, scaled up by $\frac{N}{r}$.

\subsection{Proof of Theorem~\ref{thm:etf}}

In this section we will prove the necessary and sufficient condition for the Gram matrix of an ETF to lie in $\fE_4^N$.
Let $\bv_1, \dots, \bv_N \in \RR^r$ form an ETF, let $\bV \in \RR^{r \times N}$ have the $\bv_i$ as its columns, let $\bv = \vec(\bV)$ be the concatenation of $\bv_1, \dots, \bv_N$, and let $\bX = \bV^\top \bV = \Gram(\bv_1, \dots, \bv_N)$.
Then, our result is that $\bX \in \fE_4^N$ if and only if $N < \frac{r(r + 1)}{2}$ or $r = 1$.
If $r = 1$, then each $\bv_i$ is a scalar equal to $\pm 1$, so $\bX \in \fC^N$.
Thus, it suffices to restrict our attention to $r > 1$.

First, we recall a classical result on equiangular lines (not necessarily forming a tight frame) which shows that this result only excludes one extremal case.
We also include its elegant proof, since similar ideas will be involved in our argument.
\begin{proposition}[Gerzon Bound \cite{lemmens:91}]
    If $\bv_1, \dots, \bv_N \in \SS^{r-1}$ and $|\la \bv_i, \bv_j \ra| = \alpha < 1$ for all $i, j \in [N]$ with $i \neq j$, then $N \leq \frac{r(r + 1)}{2}$.
    \label{prop:etf-gerzon-bound}
\end{proposition}
\begin{proof}
    For all $i \neq j$, $\la \bv_i\bv_i^\top, \bv_j\bv_j^\top \ra = \alpha^2$.
    Thus,
    \begin{equation}
        \Gram(\bv_1\bv_1^\top, \dots, \bv_N\bv_N^\top) = (1 - \alpha^2) \bm I_N + \alpha^2 \one\one^\top,
    \end{equation}
    which is non-singular.
    The $\bv_i\bv_i^\top$ are then linearly independent, so $N \leq \dim(\RR^{r \times r}_{\sym}) = \frac{r(r + 1)}{2}$.
\end{proof}
\noindent
By Theorem~\ref{thm:sos-pataki}, the negative direction of Theorem~\ref{thm:etf} immediately follows: if $N = \frac{r(r + 1)}{2}$, then the $\bv_i\bv_i^\top$ span $\RR^{r \times r}_{\sym}$, so by Proposition~\ref{prop:li-tam} $\bX$ is an extreme point of $\fE_2^N$, thus $\bX$ cannot belong to $\fE_4^N$ unless $\rank(\bX) = 1$, which is a contradiction if $r > 1$.

The positive direction with $r > 1$ is the more difficult part of the result.
We proceed by explicitly constructing $\bM \in \sB(N, r)$ with $\bv^\top \bM \bv = N^2$.
The construction is optimistic: we consider the simplest possible choice for $\bM$ respecting the constraint of Lemma~\ref{lem:subspace-constraint}.
The Lemma forces $\bM = \bv\bv^\top + \bM^\prime$ where $\bM^\prime \succeq 0$ with all of its eigenvectors with positive eigenvalue lying in the subspace $V_{\sym}^\prime$.
We then simply choose $\bM^\prime$ to equal a constant multiple of $\bP_{V_{\sym}^\prime}$.
Choosing the constant factor such that $\Tr(\bM) = rN$, we obtain the candidate
\begin{equation}
    \label{eq:etf-M-def}
    \bM \colonequals \bv\bv^\top + \frac{(r - 1)N}{\frac{r(r + 1)}{2} - N} \bP_{V_{\sym}^\prime}.
\end{equation}
If we could show that $\bM_{[ii]} = \bm I_r$ and $\bM_{[ij]}^\top = \bM_{[ij]}$ for all $i, j \in [N]$, then the proof would be complete.

Surprisingly, the naive construction \eqref{eq:etf-M-def} does satisfy these properties.
This may be verified by calculating $\bP_{V_{\sym}^\prime}$ explicitly, a calculation we perform in detail in Appendix~\ref{app:projectors} but outline briefly here.
Recall that
\begin{equation}
    V_{\sym}^\prime \colonequals \left\{ \vec(\bS\bV) : \bS \in \RR^{r \times r}_{\sym}, \bv_i^\top \bS \bv_i = 0 \text{ for } i \in [N] \right\} \subset \RR^{rN}.
\end{equation}
The basic idea is then to write $\bP_{V_{\sym}^\prime}\by$ for some $\by \in \RR^{rN}$ as $\vec(\bS\bV)$ for $\bS$ solving the least-squares optimization problem for the orthogonal projection of $\by$.
We then solve this optimization explicitly with Lagrange multipliers.
Determining the Lagrange multipliers in turn amounts to inverting the matrix $\bX^{\odot 2}$.
Fortunately, for an ETF, as noted previously in the introduction and in the proof of Proposition~\ref{prop:etf-gerzon-bound}, this matrix has a simple structure, making the calculation tractable.
In this way, we obtain formulae for the blocks of $\bP_{V_{\sym}^\prime}$ (see Corollary~\ref{cor:v-sym-prime-proj-etf}), after which it is straightforward to check that $\bM \in \sB(N, r)$.

Finally, using the relation \eqref{eq:M-Y-equiv} between the blocks of $\bM$ and the degree 4 pseudomoments, we recover the elegant formula for the degree 4 pseudomoments:
\begin{equation}
    \label{eq:etf-pe-values}
    Y_{(ij)(k\ell)} = \frac{\frac{r(r - 1)}{2}}{\frac{r(r + 1)}{2} - N}(X_{ij}X_{k\ell} + X_{ik}X_{j\ell} + X_{i\ell}X_{jk}) - \frac{r^2\left(1 - \frac{1}{N}\right)}{\frac{r(r + 1)}{2} - N}\sum_{m = 1}^N X_{im}X_{jm}X_{km}X_{\ell m}.
\end{equation}
This derivation is a rather egregious instance of ``bookkeeping for a miracle'' \cite{clark-qr}, and it certainly remains an open question to provide an intuitive explanation for why any ETF Gram matrices ought to belong to $\fE_{4}^N$ at all, or for the structure of the remarkably symmetric formula \eqref{eq:etf-pe-values}.
We remark additionally that, by the comments at the beginning of this section, the spectrum of $\bY$ described by \eqref{eq:etf-pe-values} is the same as the spectrum of $\bM$ with a constant scaling, and thus is also simple: $\bY$ equals the rank one matrix $\vec(\bX)\vec(\bX)^\top$ plus a constant multiple of the projection matrix onto the subspace $\vec(\mathsf{pert}_{\fE_2^N}(\bX))$.

\section{Applications}

\subsection{\Schlafli\ Inequalities: Theorem~\ref{thm:max-etf-7-ineqs}}

In this section we will describe the computer-assisted verification of the inequalities \eqref{eq:max-etf-7-ineqs} and some ancillary results.
First, we review a connection between ETFs and strongly regular graphs (SRGs).
(In fact, there are two distinct correspondences between ETFs and SRGs: the one we will use applies to arbitrary ETFs and is described in \cite{fickus:15}, while the other applies only to ETFs with a certain additional symmetry and is described in \cite{fickus:16:centroidal}.)

\begin{definition}
    A graph $G = (V, E)$ is a \emph{strongly regular graph with parameters $(v, k, \lambda, \mu)$}, abbreviated $\srg(v, k, \lambda, \mu)$, if $|V| = v$, $G$ is $k$-regular, every $x, y \in V$ that are adjacent have $\lambda$ common neighbors, and every $x, y \in V$ that are not adjacent have $\mu$ common neighbors.
\end{definition}

\begin{proposition}[Theorem 3.1 of \cite{fickus:15}]
    Let $\bv_1, \dots, \bv_N \in \RR^r$ form an ETF with $N > r$, suppose that for all $i \in [N] \setminus \{1\}$ we have $\la \bv_1, \bv_i \ra > 0$, and let $\bX = \Gram(\bv_1, \dots, \bv_N)$.
    Define the graph $G$ on vertices in $[N] \setminus \{1\}$ where $i$ and $j$ are adjacent if and only if $\la \bv_i, \bv_j \ra > 0$.
    Then, $G$ is an $\srg(v, k, \lambda, \mu)$ with parameters
    \begin{align}
      v &= N - 1, \\
      k &= \frac{N}{2} - 1 + \left(\frac{N}{2r} - 1\right)\sqrt{\frac{r(N - 1)}{N - r}}, \\
      \mu &= \frac{k}{2}, \\
      \lambda &= \frac{3k - v - 1}{2}.
    \end{align}
\end{proposition}
\noindent
Note that the assumption that $\la \bv_1, \bv_i \ra > 0$ for all $i \neq 1$ is not a substantial restriction, since any vector in an ETF may be negated to produce another essentially equivalent ETF.

In our case, an ETF on 28 vectors in $\RR^7$ corresponds to an $\srg(27, 16, 10, 8)$.
By the result of \cite{seidel:91}, this graph is unique, so we may take it by definition to be the \Schlafli\ graph (a more natural geometric description is given in the previous reference).
Consequently, since by negating some vectors every ETF can be put into the ``canonical'' form where $\la \bv_1, \bv_i \ra > 0$ for all $i \neq 1$, we obtain the following uniqueness result.
\begin{proposition}
    Let $\bv_1, \dots, \bv_{28}$ and $\bw_1, \dots, \bw_{28}$ be two ETFs in $\RR^7$.
    Then, there exist signs $1 = s_1, s_2, \dots, s_{28} \in \{\pm 1\}$ and $\bQ \in \sO(7)$ such that $\bw_i = s_i\bQ \bv_i$ for each $i \in [28]$.
\end{proposition}

Since if $\bX \in \fE_4^N$ then $\bD\bX\bD \in \fE_4^N$ for any $\bD = \diag(\bd)$ with $\bd \in \{\pm 1\}^N$, it suffices to fix a single ETF of 28 vectors in $\RR^7$ and check \eqref{eq:max-etf-7-ineqs}, and the result will follow for all ETFs of the same dimensions.
Thus, let us fix $\bv_1, \dots, \bv_{28} \in \RR^7$ forming an ETF with $\la \bv_1, \bv_i \ra > 0$ for all $i \neq 1$, and let $\bZ = \Gram(\bv_1, \dots, \bv_N)$.
Let $G$ be the graph on $[28]$ where $i$ and $j$ are adjacent if $\la \bv_i, \bv_j \ra > 0$, so that $G$ is the \Schlafli\ graph with one extra vertex added that is attached to every other vertex.
We will write $G|_{S}$ for the subgraph induced by $G$ on the set of vertices $S$.

We show \eqref{eq:max-etf-7-ineqs} by producing a $\bm 0 \preceq \bA \in \RR^{N^2 \times N^2}$ such that for any $\bY$ a degree 4 pseudomoment matrix extending some $\bX$ a degree 2 pseudomoment matrix,
\begin{equation}
    0 \leq \la \bA, \bY \ra = 112 - \sum_{1 \leq i < j \leq 28} \sgn(Z_{ij})X_{ij}.
    \label{eq:schlafli-pf-1}
\end{equation}
The construction of $\bA$ is based on studying the results of numerical experiments.
We identify the constants appearing in $\bA$ as
\begin{align}
  \gamma_1 &\colonequals \frac{1}{126}, \\  
  \gamma_2 &\colonequals \frac{1}{36}, \\ 
  \kappa_1 &\colonequals \frac{2}{9}, \\ 
  \kappa_2 &\colonequals \frac{1}{28}. 
\end{align}
With this, we define
\begin{align}
    A_{(ij)(k\ell)} &\colonequals \left\{
        \begin{array}{lcl}
          0 & : & |\{i, j, k, \ell\}| = 4, \\
          -\sgn(Z_{k\ell})\gamma_1 & : & i = j, k \neq \ell, \\
          \gamma_2 & : & i = k, j \neq \ell, |E(G|_{\{i, j, \ell\}})| = 0, \\
          \gamma_2 & : & i = k, j \neq \ell, |E(G|_{\{i, j, \ell\}})| = 2, i \sim j, i \sim \ell, \\
          -\gamma_2 & : & i = k, j \neq \ell, |E(G|_{\{i, j, \ell\}})| = 2, j \sim \ell, \\
          0 & : & i = k, j \neq \ell, |E(G|_{\{i, j, \ell\}})| \in \{1, 3\}, \\
          -\sgn(Z_{i\ell})\gamma_1 & : & i = j = k, i \neq \ell, \\
          \kappa_1 & : & i = k, j = \ell, i \neq j, \\
          \kappa_2 & : & i = j, k = \ell.
        \end{array}\right. \\
  A_{(i_1i_2)(i_3i_4)} &= A_{(i_{\pi(1)}i_{\pi(2)})(i_{\pi(3)}i_{\pi(4)})} \text{ for } \bm i \in [N]^4, \pi \in \mathsf{Sym}(4).
\end{align}
We then perform a computer verification that $\bA \succeq \bm 0$ using the \texttt{SageMath} software package for symbolic calculation of a Cholesky decomposition.
Verifying that the equality of \eqref{eq:schlafli-pf-1} holds is straightforward by counting the occurrences of various terms in $\la \bA, \bX \ra$.
Accompanying code for reproducing the verification is available online.\footnote{See the second author's webpage at \url{http://www.kunisky.com/publications/deg-4-elliptope/}.}
Of course, this proof technique is rather unsatisfying, and it is an open problem to provide a more principled description of $\bA$ and a conceptual proof of its positive semidefiniteness (both for this specific case and for the general case of maximal ETFs for any dimensions they may exist in).

\subsection{Complexity of Parity Inequalities}

In this section, we give the straightforward argument behind our proof of Corollary~\ref{cor:laurent}, a partial reproduction of the result of Laurent given in Proposition~\ref{prop:laurent}.
The matrix $\bX^{(N)}$ described there is the Gram matrix of the following type of ETF.

\begin{definition}
    The \emph{simplex ETF} with parameter $N \geq 3$ is an ETF of $N$ vectors in $\RR^r$ with $r = N - 1$, whose vectors point to the vertices of an equilateral simplex whose barycenter lies at the origin and whose vertices are unit distance from this barycenter.
    The coherence of the simplex ETF is $\alpha = \frac{1}{r}$, and the inner product of any two distinct vectors is $-\alpha$; that is, the Gram matrix $\bX^{(N)}$ is
    \begin{equation}
        \bX^{(N)} = \left(1 + \frac{1}{N - 1}\right)\bm I_{N} - \frac{1}{N - 1}\one\one^\top.
    \end{equation}
\end{definition}
\noindent
When $N = 3$ and $r = 2$, then the simplex ETF is maximal, so $\bX^{(3)} \notin \fE_4^3$, but for all $N > 3$ we do have $\bX^{(N)} \in \fE_4^N$, and the extending degree 4 pseudomoment matrix may be computed directly from \eqref{eq:etf-pe-values} (especially simple in this case since the terms $X_{ij}$ appearing in the summations only take on two different values), completing the proof of the Corollary.

We remark that, in our previous results on ETFs, the only technical calculation was that of the projection matrix $\bP_{V^\prime_{\sym}}$, and the only particularly novel idea required was the optimistic construction of the Gram vector witness \eqref{eq:etf-M-def}.
In contrast, the original argument of \cite{laurent:03} uses some rather powerful machinery from the general theory of association schemes and analytic identities for hypergeometric functions.
We thus hope that our approach can be extended to view the higher-degree pseudomoment matrices used for the full result of Proposition~\ref{prop:laurent} as Gram matrices as well, replacing these technicalities with simpler considerations of the geometry of the simplex ETFs.

\section*{Acknowledgements}

We thank Jess Banks, Nicolas Boumal, Didier Henrion, Aida Khajavirad, Jean-Bernard Lasserre, Dustin Mixon, Cristopher Moore, and participants in the 2018 Princeton Day of Optimization for useful discussions.

\bibliographystyle{plain}
\bibliography{main}

\begin{thebibliography}{10}

\bibitem{abbe:14}
Emmanuel Abbe, Afonso~S Bandeira, Annina Bracher, and Amit Singer.
\newblock Decoding binary node labels from censored edge measurements: Phase
  transition and efficient recovery.
\newblock {\em IEEE Transactions on Network Science and Engineering},
  1(1):10--22, 2014.

\bibitem{alaoui:18}
Ahmed~El Alaoui, Florent Krzakala, and Michael~I Jordan.
\newblock Fundamental limits of detection in the spiked wigner model.
\newblock {\em arXiv preprint arXiv:1806.09588}, 2018.

\bibitem{aubrun:17:book}
Guillaume Aubrun and Stanis{\l}aw~J Szarek.
\newblock {\em Alice and {Bob} Meet {Banach}: The Interface of Asymptotic
  Geometric Analysis and Quantum Information Theory}, volume 223.
\newblock American Mathematical Soc., 2017.

\bibitem{bandeira:15}
Afonso~S Bandeira, Yutong Chen, and Amit Singer.
\newblock Non-unique games over compact groups and orientation estimation in
  cryo-em.
\newblock {\em arXiv preprint arXiv:1505.03840}, 2015.

\bibitem{barahona:82}
Francisco Barahona.
\newblock On the computational complexity of ising spin glass models.
\newblock {\em Journal of Physics A: Mathematical and General}, 15(10):3241,
  1982.

\bibitem{barak:16}
Boaz Barak, Samuel~B Hopkins, Jonathan Kelner, Pravesh Kothari, Ankur Moitra,
  and Aaron Potechin.
\newblock A nearly tight sum-of-squares lower bound for the planted clique
  problem.
\newblock In {\em Foundations of Computer Science (FOCS), 2016 IEEE 57th Annual
  Symposium on}, pages 428--437. IEEE, 2016.

\bibitem{barak:14:survey}
Boaz Barak and David Steurer.
\newblock Sum-of-squares proofs and the quest toward optimal algorithms.
\newblock {\em arXiv preprint arXiv:1404.5236}, 2014.

\bibitem{ben:01}
Ahron Ben-Tal and Arkadi Nemirovski.
\newblock {\em Lectures on modern convex optimization: analysis, algorithms,
  and engineering applications}, volume~2.
\newblock Siam, 2001.

\bibitem{benedetto:03}
John~J Benedetto and Matthew Fickus.
\newblock Finite normalized tight frames.
\newblock {\em Advances in Computational Mathematics}, 18(2-4):357--385, 2003.

\bibitem{bengtsson:17:book}
Ingemar Bengtsson and Karol {\.Z}yczkowski.
\newblock {\em Geometry of quantum states: an introduction to quantum
  entanglement}.
\newblock Cambridge University Press, 2017.

\bibitem{blekherman:12}
Grigoriy Blekherman, Pablo~A Parrilo, and Rekha~R Thomas.
\newblock {\em Semidefinite optimization and convex algebraic geometry}.
\newblock SIAM, 2012.

\bibitem{boyd:04}
Stephen Boyd and Lieven Vandenberghe.
\newblock {\em Convex optimization}.
\newblock Cambridge university press, 2004.

\bibitem{burer:03}
Samuel Burer and Renato~DC Monteiro.
\newblock A nonlinear programming algorithm for solving semidefinite programs
  via low-rank factorization.
\newblock {\em Mathematical Programming}, 95(2):329--357, 2003.

\bibitem{cameron:80}
Peter~J Cameron.
\newblock 6-transitive graphs.
\newblock {\em Journal of Combinatorial Theory, Series B}, 28(2):168--179,
  1980.

\bibitem{casazza:12}
Peter~G Casazza and Gitta Kutyniok.
\newblock {\em Finite frames: Theory and applications}.
\newblock Springer, 2012.

\bibitem{casazza:08}
Peter~G Casazza, Dan Redmond, and Janet~C Tremain.
\newblock Real equiangular frames.
\newblock In {\em Information Sciences and Systems, 2008. CISS 2008. 42nd
  Annual Conference on}, pages 715--720. IEEE, 2008.

\bibitem{chen:12}
Lin Chen and Dragomir~{\v{Z}} Djokovi{\'c}.
\newblock Qubit-qudit states with positive partial transpose.
\newblock {\em Physical Review A}, 86(6):062332, 2012.

\bibitem{chudnovsky:05}
Maria Chudnovsky and Paul~D Seymour.
\newblock The structure of claw-free graphs.
\newblock {\em Surveys in combinatorics}, 327:153--171, 2005.

\bibitem{clark-qr}
Pete~L Clark.
\newblock {Quadratic Reciprocity II: The Proofs (Lecture Notes)}, 2009.
\newblock URL: \url{http://math.uga.edu/~pete/NT2009qrproof.pdf}. Last visited
  on 2018/05/21. Quotation: ``Working through this proof feels a little bit
  like being an accountant who has been assigned to carefully document a
  miracle''.

\bibitem{colbourn:06}
Charles~J Colbourn and Jeffrey~H Dinitz.
\newblock {\em Handbook of combinatorial designs}.
\newblock CRC press, 2006.

\bibitem{deshpande:15}
Yash Deshpande and Andrea Montanari.
\newblock Improved sum-of-squares lower bounds for hidden clique and hidden
  submatrix problems.
\newblock In {\em Conference on Learning Theory}, pages 523--562, 2015.

\bibitem{deza:09:book}
M.M. Deza and M.~Laurent.
\newblock {\em Geometry of Cuts and Metrics}.
\newblock Algorithms and Combinatorics. Springer Berlin Heidelberg, 2009.

\bibitem{fawzi:16:2}
Hamza Fawzi and Pablo~A Parrilo.
\newblock Self-scaled bounds for atomic cone ranks: applications to nonnegative
  rank and cp-rank.
\newblock {\em Mathematical Programming}, 158(1-2):417--465, 2016.

\bibitem{fawzi:16}
Hamza Fawzi, James Saunderson, and Pablo~A Parrilo.
\newblock Sparse sums of squares on finite abelian groups and improved
  semidefinite lifts.
\newblock {\em Mathematical Programming}, 160(1-2):149--191, 2016.

\bibitem{feingold:62}
David~G Feingold, Richard~S Varga, et~al.
\newblock Block diagonally dominant matrices and generalizations of the
  gerschgorin circle theorem.

\bibitem{fickus:16:centroidal}
Matthew Fickus, John Jasper, Dustin~G Mixon, Jesse~D Peterson, and Cody~E
  Watson.
\newblock Equiangular tight frames with centroidal symmetry.
\newblock {\em Applied and Computational Harmonic Analysis}, 2016.

\bibitem{fickus:15:tables}
Matthew Fickus and Dustin~G Mixon.
\newblock Tables of the existence of equiangular tight frames.
\newblock {\em arXiv preprint arXiv:1504.00253}, 2015.

\bibitem{fickus:15}
Matthew Fickus and Cody~E Watson.
\newblock Detailing the equivalence between real equiangular tight frames and
  certain strongly regular graphs.
\newblock In {\em Wavelets and Sparsity XVI}, volume 9597, page 959719.
  International Society for Optics and Photonics, 2015.

\bibitem{goemans:95}
Michel~X Goemans and David~P Williamson.
\newblock Improved approximation algorithms for maximum cut and satisfiability
  problems using semidefinite programming.
\newblock {\em Journal of the ACM (JACM)}, 42(6):1115--1145, 1995.

\bibitem{grigoriev:01:2}
Dima Grigoriev.
\newblock Complexity of positivstellensatz proofs for the knapsack.
\newblock {\em computational complexity}, 10(2):139--154, 2001.

\bibitem{grigoriev:01}
Dima Grigoriev.
\newblock Linear lower bound on degrees of positivstellensatz calculus proofs
  for the parity.
\newblock {\em Theoretical Computer Science}, 259(1-2):613--622, 2001.

\bibitem{grigoriev:02}
Dima Grigoriev, Edward~A Hirsch, and Dmitrii~V Pasechnik.
\newblock Complexity of semi-algebraic proofs.
\newblock In {\em Annual Symposium on Theoretical Aspects of Computer Science},
  pages 419--430. Springer, 2002.

\bibitem{grothendieck:56}
Alexandre Grothendieck.
\newblock {\em R{\'e}sum{\'e} de la th{\'e}orie m{\'e}trique des produits
  tensoriels topologiques}.
\newblock Soc. de Matem{\'a}tica de S{\~a}o Paulo, 1956.

\bibitem{hildebrand:07}
Roland Hildebrand.
\newblock Positive partial transpose from spectra.
\newblock {\em Physical Review A}, 76(5):052325, 2007.

\bibitem{hopkins:18}
Samuel~B Hopkins, Pravesh Kothari, Aaron~Henry Potechin, Prasad Raghavendra,
  and Tselil Schramm.
\newblock On the integrality gap of degree-4 sum of squares for planted clique.
\newblock {\em ACM Transactions on Algorithms (TALG)}, 14(3):28, 2018.

\bibitem{hopkins:17}
Samuel~B Hopkins, Pravesh~K Kothari, Aaron Potechin, Prasad Raghavendra, Tselil
  Schramm, and David Steurer.
\newblock The power of sum-of-squares for detecting hidden structures.
\newblock In {\em Foundations of Computer Science (FOCS), 2017 IEEE 58th Annual
  Symposium on}, pages 720--731. IEEE, 2017.

\bibitem{horedecki:96}
M~Horedecki, P~Horodecki, and R~Horodecki.
\newblock Separability of mixed states: necessary and sufficient conditions
  phys.
\newblock {\em Lett. A}, 223:1--8, 1996.

\bibitem{jaeger:07:book}
Gregg Jaeger.
\newblock {\em Quantum information}.
\newblock Springer, 2007.

\bibitem{johnston:18}
Nathaniel Johnston and Everett Patterson.
\newblock The inverse eigenvalue problem for entanglement witnesses.
\newblock {\em Linear Algebra and its Applications}, 550:1--27, 2018.

\bibitem{karp:72}
Richard~M Karp.
\newblock Reducibility among combinatorial problems.
\newblock In {\em Complexity of computer computations}, pages 85--103.
  Springer, 1972.

\bibitem{khot:07}
Subhash Khot, Guy Kindler, Elchanan Mossel, and Ryan O’Donnell.
\newblock Optimal inapproximability results for max-cut and other 2-variable
  csps?
\newblock {\em SIAM Journal on Computing}, 37(1):319--357, 2007.

\bibitem{khot:11}
Subhash Khot and Assaf Naor.
\newblock Grothendieck-type inequalities in combinatorial optimization.
\newblock {\em arXiv preprint arXiv:1108.2464}, 2011.

\bibitem{khot:05}
Subhash Khot and Nisheeth~K Vishnoi.
\newblock On the unique games conjecture.
\newblock In {\em FOCS}, volume~5, page~3, 2005.

\bibitem{kuvcera:74}
Vladim{\'\i}r Ku{\v{c}}era.
\newblock The matrix equation ax+xb=c.
\newblock {\em SIAM Journal on Applied Mathematics}, 26(1):15--25, 1974.

\bibitem{lasserre:01}
Jean~B Lasserre.
\newblock Global optimization with polynomials and the problem of moments.
\newblock {\em SIAM Journal on Optimization}, 11(3):796--817, 2001.

\bibitem{laurent:03}
Monique Laurent.
\newblock Lower bound for the number of iterations in semidefinite hierarchies
  for the cut polytope.
\newblock {\em Mathematics of operations research}, 28(4):871--883, 2003.

\bibitem{laurent:09}
Monique Laurent.
\newblock Sums of squares, moment matrices and optimization over polynomials.
\newblock In {\em Emerging applications of algebraic geometry}, pages 157--270.
  Springer, 2009.

\bibitem{laurent:95}
Monique Laurent and Svatopluk Poljak.
\newblock On a positive semidefinite relaxation of the cut polytope.
\newblock {\em Linear Algebra and its Applications}, 223:439--461, 1995.

\bibitem{laurent:96}
Monique Laurent and Svatopluk Poljak.
\newblock On the facial structure of the set of correlation matrices.
\newblock {\em SIAM Journal on Matrix Analysis and Applications},
  17(3):530--547, 1996.

\bibitem{leinaas:10}
Jon~Magne Leinaas, Jan Myrheim, and Per~{\O}yvind Sollid.
\newblock Numerical studies of entangled positive-partial-transpose states in
  composite quantum systems.
\newblock {\em Physical Review A}, 81(6):062329, 2010.

\bibitem{lemmens:91}
Petrus~WH Lemmens, Johan~J Seidel, and JA~Green.
\newblock Equiangular lines.
\newblock In {\em Geometry and Combinatorics}, pages 127--145. Elsevier, 1991.

\bibitem{lewenstein:00}
Maciej Lewenstein, B~Kraus, JI~Cirac, and P~Horodecki.
\newblock Optimization of entanglement witnesses.
\newblock {\em Physical Review A}, 62(5):052310, 2000.

\bibitem{li:94}
Chi-Kwong Li and Bit-Shun Tam.
\newblock A note on extreme correlation matrices.
\newblock {\em SIAM Journal on Matrix Analysis and Applications},
  15(3):903--908, 1994.

\bibitem{liuset}
Ya-Feng Liu.
\newblock Set-completely-positive representations and cuts for the max-cut
  polytope and the unit modulus lifting.

\bibitem{meka:15}
Raghu Meka, Aaron Potechin, and Avi Wigderson.
\newblock Sum-of-squares lower bounds for planted clique.
\newblock In {\em Proceedings of the forty-seventh annual ACM symposium on
  Theory of computing}, pages 87--96. ACM, 2015.

\bibitem{montanari:16}
Andrea Montanari and Subhabrata Sen.
\newblock Semidefinite programs on sparse random graphs and their application
  to community detection.
\newblock In {\em Proceedings of the forty-eighth annual ACM symposium on
  Theory of Computing}, pages 814--827. ACM, 2016.

\bibitem{nesterov:98}
Yurii Nesterov.
\newblock Semidefinite relaxation and nonconvex quadratic optimization.
\newblock {\em Optimization methods and software}, 9(1-3):141--160, 1998.

\bibitem{panchenko:13}
Dmitry Panchenko.
\newblock {\em The Sherrington-Kirkpatrick model}.
\newblock Springer Science \& Business Media, 2013.

\bibitem{pataki:98}
G{\'a}bor Pataki.
\newblock On the rank of extreme matrices in semidefinite programs and the
  multiplicity of optimal eigenvalues.
\newblock {\em Mathematics of operations research}, 23(2):339--358, 1998.

\bibitem{peres:96}
Asher Peres.
\newblock Separability criterion for density matrices.
\newblock {\em Physical Review Letters}, 77(8):1413, 1996.

\bibitem{perry:18}
Amelia Perry, Alexander~S Wein, Afonso~S Bandeira, Ankur Moitra, et~al.
\newblock Optimality and sub-optimality of pca i: Spiked random matrix models.
\newblock {\em The Annals of Statistics}, 46(5):2416--2451, 2018.

\bibitem{poljak:95}
Svatopluk Poljak and Zsolt Tuza.
\newblock Maximum cuts and large bipartite subgraphs.
\newblock 1995.

\bibitem{raghavendra:08}
Prasad Raghavendra.
\newblock Optimal algorithms and inapproximability results for every {CSP}?
\newblock In {\em Proceedings of the fortieth annual ACM symposium on Theory of
  computing}, pages 245--254. ACM, 2008.

\bibitem{raghavendra:18}
Prasad Raghavendra, Tselil Schramm, and David Steurer.
\newblock High-dimensional estimation via sum-of-squares proofs.
\newblock {\em arXiv preprint arXiv:1807.11419}, 2018.

\bibitem{schoenebeck:08}
Grant Schoenebeck.
\newblock Linear level lasserre lower bounds for certain k-csps.
\newblock In {\em Foundations of Computer Science, 2008. FOCS'08. IEEE 49th
  Annual IEEE Symposium on}, pages 593--602. IEEE, 2008.

\bibitem{seidel:91}
Johan~Jacob Seidel.
\newblock Strongly regular graphs with (—1, 1, 0) adjacency matrix having
  eigenvalue 3.
\newblock In {\em Geometry and Combinatorics}, pages 26--43. Elsevier, 1991.

\bibitem{tropp:07}
M{\'a}ty{\'a}s~A Sustik, Joel~A Tropp, Inderjit~S Dhillon, and Robert~W
  Heath~Jr.
\newblock On the existence of equiangular tight frames.
\newblock {\em Linear Algebra and its applications}, 426(2-3):619--635, 2007.

\bibitem{szarek:06}
Stanis{\l}aw~J Szarek, Ingemar Bengtsson, and Karol {\.Z}yczkowski.
\newblock On the structure of the body of states with positive partial
  transpose.
\newblock {\em Journal of Physics A: Mathematical and General}, 39(5):L119,
  2006.

\bibitem{trevisan:12:survey}
Luca Trevisan.
\newblock On {Khot's} {Unique} {Games} {Conjecture}.
\newblock {\em Bulletin (New Series) of the American Mathematical Society},
  49(1), 2012.

\bibitem{welch:74}
Lloyd Welch.
\newblock Lower bounds on the maximum cross correlation of signals (corresp.).
\newblock {\em IEEE Transactions on Information theory}, 20(3):397--399, 1974.

\end{thebibliography}

\clearpage
\appendix

\section{Pseudomoment Reductions for Sum-of-Squares over $\{\pm 1\}^N$}
\label{app:sos-reductions}

In this appendix, we explain some standard reductions for the degree $d$ SOS relaxation of the problem
\begin{equation}
    \M(\bW) = \max_{\bx \in \{\pm 1\}^N}\bx^\top \bW \bx = \max_{\substack{\bx \in \RR^N \\ x_i^2 - 1 = 0 \text{ for } i \in [N]}} \sum_{i = 1}^N \sum_{j = 1}^N W_{ij}x_ix_j.
\end{equation}
The second expression above writes $\M(\bW)$ as a polynomial optimization problem, so the standard machinery of SOS optimization (see e.g.\ \cite{lasserre:01,laurent:09}) may be applied to formulate the degree $d$ relaxation.
We first describe the decision variable of this relaxation.

\begin{definition}
    \label{def:complete-pm}
    Let $d$ be an even positive integer.
    Then, $M^{(d)} \subset \RR^{N^{\leq d / 2} \times N^{\leq d / 2}}$ is the set of \emph{degree $d$ complete pseudomoment matrices}, consisting of $\bZ$ whose row and column indices we identify with the set $[N]^{\leq d / 2}$ ordered first by ascending length and then lexicographically and satisfying the following properties.
    \begin{enumerate}
    \item $\bZ \succeq \bm 0$.
    \item $Z_{\bs \bm t}$ depends only on $\odd(\bs \circ \bm t)$.
    \item $Z_{\bs \bm t} = 1$ whenever $\odd(\bs \circ \bm t) = \emptyset$.
    \end{enumerate}
\end{definition}
\noindent
We then define the usual formulation of the degree $d$ SOS relaxation of $\M(\bW)$ in the following way.

\begin{definition}
    Let $d$ be an even positive integer.
    The \emph{degree $d$ SOS relaxation of $\M(\bW)$} is the optimization problem
    \begin{equation}
        \SOS_d(\bW) \colonequals \max_{\bZ \in M^{(d)}} \sum_{i = 1}^N \sum_{j = 1}^N W_{ij} Z_{(i)(j)},
        \label{eq:sos-d}
    \end{equation}
    where $(i) \in [N]$ and $(j) \in [N]$ are interpreted as strings of length 1.
\end{definition}

The result we will prove in this appendix is that the pseudomoment matrices of Definition~\ref{def:complete-pm} can be truncated to just the minor indexed by $[N]^{d/2} \times [N]^{d/2}$ without affecting the optimization problem \eqref{eq:sos-d}.
First, we make the simple observation that, by Condition 2 of Definition~\ref{def:complete-pm}, the objective function of $\SOS_d(\bW)$ may be rewritten in terms of this minor.

\begin{definition}
    For any $N \geq 1$, let $\be_k \in [N]^k$ be the string of length $k$ with all entries equal to the symbol $1 \in [N]$.
\end{definition}
\begin{proposition}
    For each $d$ an even positive integer,
    \begin{equation}
        \SOS_d(\bW) = \max_{\bZ \in M^{(d)}} \sum_{i = 1}^N \sum_{j = 1}^N W_{ij} Z_{(\be_{d / 2 - 1} \circ (i))(\be_{d / 2 - 1} \circ (j))}.
    \end{equation}
\end{proposition}
\noindent
We next show that it does not matter whether we define the set of minors of $\bZ \in M^{(d)}$ indexed by $[N]^{d / 2} \times [N]^{d / 2}$ by truncating $\bZ$ or by applying the constraints of Definition~\ref{def:complete-pm} to only a subset of strings $[N]^{d / 2} \subset [N]^{\leq d / 2}$, as we did in the main text in Definition~\ref{def:truncated-pm}.

\begin{definition}
    Let $d$ be an even positive integer.
    Then, $\tilde{M}^{(d)} \subset \RR^{[N]^{d / 2} \times [N]^{d / 2}}$ is the set of \emph{degree $d$ truncated pseudomoment matrices}, consisting of $\tilde{\bZ}$ satisfying the properties of Definition~\ref{def:complete-pm} but only for strings of length exactly $d / 2$, that is, for $\bs, \bm t \in [N]^{d / 2}$.
\end{definition}

\begin{proposition}
    $\tilde{M}^{(d)}$ is equal to the set of $\tilde{\bZ}$ occurring as the minor indexed by $[N]^{d / 2} \times [N]^{d / 2}$ of $\bZ \in M^{(d)}$.
\end{proposition}
\begin{proof}
    Clearly if $\bZ \in M^{(d)}$ then the $[N]^{d / 2} \times [N]^{d / 2}$ minor of $\bZ$ will belong to $\tilde{M}^{(d)}$.
    Therefore, it suffices to prove that for any $\tilde{\bZ} \in \tilde{M}^{(d)}$, there exists $\bZ \in M^{(d)}$ such that $\tilde{\bZ}$ is the $[N]^{d / 2} \times [N]^{d / 2}$ minor of $\bZ$.

    We define the entries of $\bZ$ to be
    \begin{equation}
        Z_{\bs \bm t} \colonequals \tilde{Z}_{(\bs \circ \be_{d / 2 - |\bs|}) (\bm t \circ \be_{d / 2 - |\bm t|})}.
    \end{equation}
    (Intuitively, this construction is in analogy to the possibility of assuming without loss of generality that $x_1 = 1$ in the optimization defining $\M(\bW)$.)
    By construction, $\tilde{\bZ}$ is the necessary minor of $\bZ$.
    We have
    \begin{equation}
        \odd((\bs \circ \be_{d / 2 - |\bs|}) \circ (\bm t \circ \be_{d / 2 - |\bm t|})) = \odd(\bs \circ \bm t \circ \be_{|\bs| + |\bm t|}),
        \label{eq:pm-reduction-pf-1}
    \end{equation}
    and $|\bm s| + |\bm t| \equiv |\odd(\bs \circ \bt)| \pmod{2}$, so $Z_{\bs \bm t}$ is a function of only $\odd(\bs \circ \bm t)$.
    Also, if $\odd(\bs \circ \bm t)= \emptyset$ then $|\bs| + |\bm t|$ must be even, so in this case the expression in \eqref{eq:pm-reduction-pf-1} also equals $\emptyset$, thus in this case $Z_{\bs \bm t} = 1$ since $\tilde{\bZ} \in \tilde{M}^{(d)}$.

    It then only remains to show that $\bZ \succeq \bm 0$ to show that $\bZ \in M^{(d)}$.
    For two strings $\bs, \bs^\prime \in [N]^{< \infty}$, let us write $\bs \leq_1 \bs^\prime$ if $|\bs| \leq |\bs^\prime|$ and $\bs^\prime = \bs \circ \be_{|\bs^\prime| - |\bs|}$.
    Suppose that $\bv \in \RR^{[N] \leq d / 2}$, then using the above definition we may write its quadratic form with $\bZ$ as
    \begin{align}
      \bv^\top \bZ \bv
      &= \sum_{\bs \in [N]^{\leq d / 2}}\sum_{\bm t \in [N]^{\leq d / 2}} Z_{\bs \bm t} v_{\bs} v_{\bm t} \nonumber \\
      &= \sum_{\bs^\prime \in [N]^{d / 2}}\sum_{\bm t^\prime \in [N]^{d / 2}}\tilde{Z}_{\bm s^\prime \bm t^\prime}\left(\sum_{\substack{\bs \in [N]^{\leq d / 2} \\ \bs \leq_1 \bs^\prime}}\sum_{\substack{\bm t \in [N]^{\leq d / 2} \\ \bm t \leq_1 \bm t^\prime}} v_{\bs} v_{\bm t}\right) \nonumber \\
      &= \sum_{\bs^\prime \in [N]^{d / 2}}\sum_{\bm t^\prime \in [N]^{d / 2}}\tilde{Z}_{\bm s^\prime \bm t^\prime}\left(\sum_{\substack{\bs \in [N]^{\leq d / 2} \\ \bs \leq_1 \bs^\prime}} v_{\bs}\right)\left(\sum_{\substack{\bm t \in [N]^{\leq d / 2} \\ \bm t \leq_1 \bm t^\prime}} v_{\bm t}\right) \nonumber \\
      &\geq 0,
    \end{align}
    where the last inequality follows because $\tilde{\bZ} \succeq \bm 0$.
    Thus, $\bZ \in M^{(d)}$, completing the proof.
\end{proof}

The result we were interested in then follows, that $\SOS_d(\bW)$ may equivalently be defined in terms of optimization over the truncated pseudomoment matrices $\tilde{M}^{(d)}$.
\begin{corollary}
    For each $d$ an even positive integer,
    \begin{equation}
        \max_{\tilde{\bZ} \in \tilde{M}^{(d)}} \sum_{i = 1}^N \sum_{j = 1}^N W_{ij} \tilde{Z}_{(\be_{d / 2 - 1} \circ (i))(\be_{d / 2 - 1} \circ (j))}.
    \end{equation}
\end{corollary}

\section{Proofs of Structural Results on $\sB(N, r)$}
\label{app:BNr-structure}

\subsection{Proof of Proposition~\ref{prop:blocksym-structure}}

Let $\bM \in \sB(N, r)$.
To obtain the spectral bound on the blocks $\|\bM_{[ij]}\| \leq 1$, note that the claim is trivial for $i = j$, so let us fix $i, j \in [N]$ with $i \neq j$ and denote $\bS \colonequals \bM_{[ij]} \in \RR^{r \times r}_{\sym}$.
Taking a suitable minor of $\bM$, we find
\begin{equation}
    \left[\begin{array}{cc} \bm I_r & \bS \\ \bS & \bm I_r \end{array}\right] \succeq \bm 0.
\end{equation}
Taking a quadratic form with this matrix, we find that for any $\bv \in \RR^r$ with $\|\bv\|_2 = 1$,
\begin{equation}
    0 \leq \left[\begin{array}{c} \pm \bv \\ \bv \end{array}\right]^\top \left[\begin{array}{cc} \bm I_r & \bS \\ \bS & \bm I_r \end{array}\right]\left[\begin{array}{c} \pm \bv \\ \bv \end{array}\right] = 2 \pm 2\bv^\top \bS \bv,
\end{equation}
thus $|\bv^\top \bS \bv| \leq 1$, and the result follows.

From this, the bound $\|\bM\| \leq N$ follows from a simple case of the ``block Gershgorin circle theorem'' \cite{feingold:62}, which may be deduced directly as follows: suppose $\bv \in \RR^{rN}$ is the concatenation of $\bv_1, \dots, \bv_N \in \RR^r$, then
\begin{equation}
    \label{eq:pf-blocksym-structure-1}
    \bv^\top \bM \bv \leq \sum_{i = 1}^N \sum_{j = 1}^N |\bv_i^\top \bM_{[ij]} \bv_j| \leq \sum_{i = 1}^N \sum_{j = 1}^N \|\bv_i\|_2\|\bv_j\|_2 = \left(\sum_{i = 1}^N \|\bv_i\|_2\right)^2 \leq N\sum_{i = 1}^N \|\bv_i\|_2^2 = N\|\bv\|_2^2,
\end{equation}
giving the result.

For the final statement of the Proposition, if $\bM\bv = N\bv$, then all of the inequalities in \eqref{eq:pf-blocksym-structure-1} must be equalities.
For the third inequality to be an equality requires all of the $\|\bv_i\|_2$ to be equal for $i \in [N]$.
For the first inequality to be an equality requires $\bv_i^\top\bM_{[ij]}\bv_j \geq 0$ for all $i, j \in [N]$.
For the second inequality to be an equality requires $\bM_{[ij]}\bv_j = \bv_i$ for all $i, j \in [N]$, completing the proof.

\subsection{Proof of Proposition~\ref{prop:blocksym-factorization}}

Let $\bM \in \sB(N, r)$ and let $r^\prime \colonequals \rank(\bM)$.
Since $\bM$ contains $\bm I_r$ as a minor, $r^\prime \geq r$, and since $rN$ is the dimension of $\bM$, $r^\prime \leq rN$.
Then, there exists $\bU \in \RR^{r^\prime \times rN}$ such that $\bM = \bU^\top \bU$.
Let us expand in blocks
\begin{equation}
    \bU = \left[\begin{array}{cccc} \bU_1 & \bU_2 & \cdots & \bU_N \end{array}\right],
\end{equation}
for $\bU_i \in \RR^{r^\prime \times r}$.
Then, $\bU_i^\top \bU_i = \bM_{[ii]} = \bm I_r$.

This factorization is unchanged by multiplying $\bU$ on the left by any matrix of $\sO(r^\prime)$.
Since $\bU_1$ has orthogonal columns, by choosing a suitable such multiplication we may assume without loss of generality that the columns of $\bU_1$ are the first $r$ standard basis vectors $\be_1, \dots, \be_r \in \RR^{r^\prime}$.
Equivalently,
\begin{equation}
    \bU_1 = \left[\begin{array}{c} \bm I_r \\ \bm 0 \end{array}\right] \hspace{-0.15cm}\begin{array}{l} \} \hspace{0.1cm} r \\ \} \hspace{0.1cm} r^\prime - r\end{array}.
\end{equation}

Let us expand each $\bU_i$ in blocks of the same dimensions,
\begin{equation}
    \bU_i \equalscolon \left[\begin{array}{c} \bm S_i \\ \bm R_i \end{array}\right] \hspace{-0.15cm}\begin{array}{l} \} \hspace{0.1cm} r \\ \} \hspace{0.1cm} r^\prime - r\end{array},
\end{equation}
then $\bS_1 = \bm I_r$ and $\bR_1 = \bm 0$.
We first show that the $\bS_i$ are all symmetric.
Expanding the block $\bM_{[1i]}$, we have
\begin{equation}
    \bM_{[1i]} = \bU_1^\top \bU_i = \bm S_1^\top \bm S_i + \bm R_1^\top \bm R_i = \bm S_i,
\end{equation}
and since $\bM_{[1i]}$ is symmetric, $\bS_i$ is symmetric as well.

It remains to show the relations \eqref{eq:s-r-diag} and \eqref{eq:s-r-commutator}.
For the former, we expand $\bM_{[ii]}$:
\begin{equation}
    \bm I_r = \bM_{[ii]} = \bU_i^\top \bU_i = \bm S_i^2 + \bm R_i^\top \bm R_i.
\end{equation}
For the latter, we expand $\bM_{[ij]}$ and $\bM_{[ji]}$:
\begin{equation}
    \bm 0 = \bM_{[ij]} - \bM_{[ji]} = \bU_i^\top \bU_j - \bU_j^\top \bU_i = \bm S_i \bm S_j - \bS_j\bS_i + \bm R_i^\top \bm R_j - \bR_j^\top \bR_i.
\end{equation}

\section{Proofs of Relaxation Descriptions of Theorem~\ref{thm:sos4}}
\label{app:sos4-relaxations}

\subsection{Proof of Proposition~\ref{prop:blocksym-rank-constrained}}

\paragraph{Positive direction.}
Suppose $\bv_1, \dots, \bv_N \in \RR^r$, $\bX = \Gram(\bv_1, \dots, \bv_N)$, $\bv \in \RR^{rN}$ is the concatenation of $\bv_1, \dots, \bv_N$, $\sum_{i = 1}^N \|\bv_i\|_2^2 = N$, and $\bM \in \sB(N, r)$ with $\rank(\bM) = r$ and $\bv^\top \bM \bv = N^2$.
By Proposition~\ref{prop:blocksym-structure}, $\|\bv_i\|_2 = 1$ for each $i \in [N]$ and $\bM_{[ij]} \bv_j = \bv_i$ for each $i, j \in [N]$.

Since $\bM \succeq 0$ and $\rank(\bM) = r$, there exist $\bQ_i \in \RR^{r \times r}$ such that $\bM_{[ij]} = \bQ_i^\top \bQ_j$.
Moreover, since $\bQ_i^\top \bQ_i = \bM_{[ii]} = \bm I_r$, $\bQ_i \in \sO(r)$ for each $i \in [N]$.
The above factorization is unchanged by multiplying each $\bQ_i$ on the left by an orthogonal matrix, so we may assume without loss of generality that $\bQ_1 = \bm I_r$.

Thus, $\bM_{[1i]} = \bQ_1^\top \bQ_i = \bQ_i$, which must be symmetric, so $\bQ_i$ is symmetric for each $i \in [N]$.
And, $\bM_{[ij]} = \bQ_i\bQ_j$ is also symmetric, so $\bQ_1, \dots, \bQ_N$ are a commuting family of symmetric orthogonal matrices.
Therefore, there exists some $\bQ \in \sO(r)$ and $\one = \bd_1, \dots, \bd_N \in \{\pm 1\}^r$ such that $\bQ_i = \bQ \bD_i \bQ^\top$ where $\bD_i = \diag(\bd_i)$.

We have $\bv_i = \bM_{[i1]}\bv_1 = \bQ_i \bv_1 = \bQ \bD_i \bQ^\top \bv_1$ for each $i \in [N]$.
Thus,
\begin{equation}
    \label{eq:blocksym-rank-constrained-pf-1}
    X_{ij} = \la \bv_i, \bv_j \ra = \la \bD_i \bQ^\top \bv_1, \bD_j \bQ^\top \bv_1 \ra = \la \bD_i\bD_j, \bQ^\top \bv_1\bv_1^\top \bQ \ra.
\end{equation}
Let $\bm \rho = \diag(\bQ^\top \bv_1\bv_1^\top \bQ)$, then since $\bQ^\top \bv_1\bv_1^\top \bQ \succeq \bm 0$, $\bm\rho \geq 0$, and $\sum_{i = 1}^r \rho_i = \Tr(\bQ^\top \bv_1\bv_1^\top \bQ) = 1$.
Therefore, letting $\tilde{\bd}_k \colonequals ((\bd_i)_k)_{k = 1}^N \in \{ \pm 1\}^N$, \eqref{eq:blocksym-rank-constrained-pf-1} is
\begin{align}
  X_{ij} &= \sum_{k = 1}^r \rho_k (\bd_i)_k (\bd_j)_k, \\
  \bX &= \sum_{k = 1}^r \rho_k \tilde{\bd}_k\tilde{\bd}_k^\top \in \fC^N,
\end{align}
completing the proof.

\paragraph{Negative direction.}

We have $\bm I_N \in \fC^N$ since $\bm I_N = \frac{1}{2^N}\sum_{\bx \in \{ \pm 1\}^N} \bx\bx^\top$, as each off-diagonal entry occurs an equal number of times with a positive sign as with a negative sign in the summation.
We will view $\bm I_N = \Gram(\be_1, \dots, \be_N)$, let $\bv = \sum_{i = 1}^N \be_i \otimes \be_i$ be the concatenation of the $\be_i$, and will show that if $\bM \in \sB(N, N)$ with $\bv^\top \bM \bv = N^2$, then $\rank(\bM) > N$ when $N \notin \{1, 2\} \cup 4\NN$.

Suppose otherwise, then, as in the argument above, $\bM \in \sB(N, N)$ has $\bM_{[ij]} = \bQ_i\bQ_j$ for some $\bQ_i \in \sO(N) \cap \RR^{N \times N}_{\sym}$, with $\bQ_1 = \bm I_N$, and where $\bQ_1, \dots, \bQ_N$ commute.
We may then write $\bQ_i = \bQ\bD_i\bQ^\top$ for $\bQ \in \sO(N)$ and $\bD_i = \diag(\bd_i)$ for $\bd_i \in \{ \pm 1 \}^N$.
Let us also write $\bq_1, \dots, \bq_N$ for the rows of $\bQ$, which form an orthonormal basis of $\RR^N$.

We have
\begin{equation}
    N^2 = \bv^\top \bM \bv = \sum_{i = 1}^N \sum_{j = 1}^N (\be_i \otimes \be_i)^\top \bM (\be_j \otimes \be_j) = \sum_{i = 1}^N \sum_{j = 1}^N (\bM_{[ij]})_{ij}.
\end{equation}
Since $\bM \succeq 0$ and $\diag(\bM) = \one$, all entries of $\bM$ are at most 1, so each term in this sum must equal 1, i.e.\ $(\bM_{[ij]})_{ij} = 1$ for all $i, j \in [N]$.
We then have, for any $i, j$,
\begin{equation}
    1 = (\bM_{[ij]})_{ij} = \be_i^\top \bQ \bD_i\bD_j \bQ^\top \be_j = \la \bD_i \bq_i, \bD_j \bq_j \ra,
\end{equation}
whereby $\bD_i\bq_i = \bD_j \bq_j$ for all $i, j$.
In other words, there exists some $\bq \in \RR^N$ with $\|\bq\|_2 = 1$ such that $\bD_i \bq_i = \bq$, or $\bq_i = \bD_i \bq$.
Thus, the $\bq_i$ are sign flips of a fixed vector.

On the other hand, the $\bq_i$ are the rows of $\bQ \in \sO(N)$, whose columns must also form an orthonormal basis.
Therefore, every entry of $\bq$ must have the same norm, so each entry of $\bQ$ also has equal norm; in other words, $\bQ$ is, up to a scaling depending on definitions, a Hadamard matrix with real entries \cite{colbourn:06}.
A real-valued Hadamard matrix of order $N$ can only exist when $N \in \{1, 2\} \cup 4\NN$, so under the assumptions of the Proposition this is a contradiction.

This example is simple to analyze but probably suboptimal.
In general, a suitable example for this result is a matrix $\bX \in \fC^N$ where $\rank(\bX)$ is strictly smaller than the smallest number of cut matrices $\bx_1\bx_1^\top, \dots, \bx_m\bx_m^\top$ for $\bx_i \in \{ \pm 1\}^N$ such that $\bX \in \conv(\{\bx_i\bx_i^\top\}_{i = 1}^m)$.
The latter quantity is similar to the notions of \emph{completely-positive rank} and \emph{non-negative rank}, and appears to behave counterintuitively sometimes; see \cite{fawzi:16:2,liuset} for some discussion.

\subsection{Proof of Proposition~\ref{prop:blocksym-separable}}

Suppose first that $\bX = \Gram(\bv_1, \dots, \bv_N)$ for $\bv_i \in \RR^r$ with $\sum_{i = 1}^N \|\bv_i\|_2^2 = N$, and $\bM \in \sB_{\sep}(N, r)$ such that $\bv^\top \bM \bv = N^2$.
By Proposition~\ref{prop:blocksym-structure}, $\|\bv_i\|_2 = 1$ for each $i \in [N]$.
By absorbing constants and rearranging tensor products, the condition $\bM \in \sB_{\sep}(N, r)$ may be rewritten as
\begin{equation}
    \bM = \sum_{i = 1}^m \bA_i \otimes (\bb_i\bb_i^\top)
\end{equation}
for some $\bA_i \in \RR^{N \times N}_{\sym}$ with $\bA_i \succeq \bm 0$ and such that, letting $\ba_i = \diag(\bA_i)$,
\begin{equation}
    \label{eq:blocksym-sep-pf-2}
    \sum_{i = 1}^m (\ba_i)_j\bb_i\bb_i^\top = \bm I_r
\end{equation}
for each $j \in [N]$.

Let $\bV \in \RR^{r \times N}$ have the $\bv_i$ as its columns.
Then,
\begin{equation}
    \label{eq:blocksym-sep-pf-3}
    \bv^\top \bM \bv = \sum_{i = 1}^m \sum_{j = 1}^N \sum_{k = 1}^N (\bA_i)_{jk} \la \bb_i, \bv_j \ra \la \bb_i, \bv_k \ra = \sum_{i = 1}^m \bb_i^\top \bV\bA_i \bV^\top \bb_i.
\end{equation}
We now bound $\bb_i^\top \bV\bA_i \bV^\top \bb_i$ by applying a simple matrix inequality; the rather complicated formulation below is only to handle carefully the possibility of certain diagonal entries of $\bA_i$ equaling zero.
Let $\tilde{\bA}_i$ be the maximal strictly positive definite minor of $\bA_i$, of dimension $N_i$, and let $\bw_i$ be the restriction of $\bV^\top \bb_i$ to the same indices.
Then, $\diag(\tilde{\bA}_i) > 0$.
Let $\pi_i: [N_i] \to [N]$ map the indices of this minor to the original indices, and let us define a diagonal matrix $\bD_i \in \RR^{N_i \times N_i}$ by
\begin{equation}
    (\bD_i)_{jj} \colonequals \left(\sum_{j^\prime = 1}^{N_i} \sqrt{(\tilde{\bA}_i)_{j^\prime j^\prime}} \cdot |\la \bb_i, \bv_{\pi_i(j^\prime)} \ra|\right) \frac{|\la \bb_i, \bv_{\pi_i(j)} \ra|}{\sqrt{(\tilde{\bA}_i)_{jj}}}.
\end{equation}
Then, we claim $\bD_i \succeq \bw_i\bw_i^\top$.
This is a matter of applying a weighted Cauchy-Schwarz inequality: for $\bx \in \RR^N$, we have
\begin{align}
  \bx^\top \bw_i\bw_i^\top \bx
  &= \left(\sum_{j = 1}^{N_i} x_j \la \bb_i, \bv_{\pi_i(j)} \ra \right)^2 \nonumber \\
  &\leq \left(\sum_{j^\prime = 1}^{N_i} \sqrt{(\tilde{\bA}_i)_{j^\prime j^\prime}} \cdot |\la \bb_i, \bv_{\pi_i(j^\prime)} \ra| \right)\left(\sum_{j = 1}^N \frac{|\la \bb_i, \bv_{\pi_i(j)} \ra|}{\sqrt{(\tilde{\bA}_i)_{jj}}}x_j^2 \right) \nonumber \\
  &= \sum_{j = 1}^N (\bD_i)_{jj} x_j^2.
\end{align}
Therefore,
\begin{align}
  \bb_i^\top \bV\bA_i \bV^\top \bb_i
  &= \bw_i^\top \tilde{\bA}_i \bw_i \nonumber \\
  &\leq \la \bD_i, \tilde{\bA}_i \ra \nonumber \\
  &= \left(\sum_{j = 1}^{N_i} \sqrt{(\tilde{\bA}_i)_{jj}} \cdot |\la \bb_i, \bv_{\pi_i(j)} \ra| \right)^2 \nonumber \\
  &= \left(\sum_{j = 1}^{N} \sqrt{(\ba_i)_j} \cdot |\la \bb_i, \bv_{j} \ra| \right)^2.
    \label{eq:blocksym-sep-pf-5}
\end{align}

Now, combining \eqref{eq:blocksym-sep-pf-5} with \eqref{eq:blocksym-sep-pf-2} and \eqref{eq:blocksym-sep-pf-3} and using the Cauchy-Schwarz inequality, we find
\begin{equation}
    \bv^\top \bM \bv \leq N\sum_{i = 1}^m \sum_{j = 1}^N (\ba_i)_{j} \la \bb_i, \bv_{j} \ra^2 = N\sum_{j = 1}^N \|\bv_j\|_2^2 = N^2.
    \label{eq:blocksym-sep-pf-4}
\end{equation}
Thus, the Cauchy-Schwarz inequality in \eqref{eq:blocksym-sep-pf-4} must be tight, whereby there exist $\kappa_i \geq 0$ with $\sum_{ i = 1}^m \kappa_i = 1$ such that
\begin{equation}
    (\ba_i)_j \la \bb_i, \bv_j \ra^2 = \kappa_i
    \label{eq:blocksym-sep-pf-6}
\end{equation}
for every $i \in [m]$ and $j \in [N]$.
Note in particular that if $\kappa_i > 0$ for some $i \in [m]$, then $\la \bb_i, \bv_j \ra \neq 0$ for all $j \in [N]$.
We may then define vectors $\bm\beta_{jk} \in \RR^m$ by
\begin{equation}
    (\bm\beta_{jk})_i \colonequals \left\{ \begin{array}{lcl} \sqrt{\kappa_i} \frac{\la \bb_i, \bv_j \ra}{\la \bb_i, \bv_k \ra} & : & \kappa_i > 0, \\ 0 & : & \kappa_i = 0. \end{array}\right.
\end{equation}
Then,
\begin{align}
  \|\bm\beta_{jk}\|_2^2 &= \sum_{i: \kappa_i > 0} \kappa_i \frac{\la \bb_i, \bv_j \ra^2}{\la \bb_i, \bv_k \ra^2} \nonumber \\
                        &= \sum_{i: \kappa_i > 0} (\bm a_i)_k\la \bb_i, \bv_j \ra^2 \nonumber \\
                        &\leq \sum_{i = 1}^m (\bm a_i)_k\la \bb_i, \bv_j \ra^2 \nonumber \\
  &\stackrel{\eqref{eq:blocksym-sep-pf-2}}{=} 1, \\
  \la \bm\beta_{jk}, \bm\beta_{kj} \ra &= \sum_{i: \kappa_i > 0} \kappa_i = 1.
\end{align}
Thus, in fact $\|\bm \beta_{jk}\|_2 = 1$ and $\bm\beta_{jk} = \bm\beta_{kj}$ for all $j, k \in [N]$.
This implies first that whenever $\kappa_i > 0$ then $\la \bb_i, \bv_j \ra^2$ does not depend on $j$, and second that whenever $\kappa_i = 0$ then $(\ba_i)_k \la \bb_i, \bv_j \ra^2 = 0$ for all $j, k \in [N]$.
We may assume without loss of generality that $\bA_i \neq \bm 0$, so $\ba_i \neq \bm 0$, and thus the latter implies that whenever $\kappa_i = 0$, then $\la \bb_i, \bv_j \ra^2 = 0$ for all $j \in [N]$.
Therefore, in all cases, $\la \bb_i, \bv_j \ra^2$ does not depend on $j$.

Let us write $\eta_i \colonequals \la \bb_i, \bv_j \ra^2$.
For $i$ where $\eta_i \neq 0$, by \eqref{eq:blocksym-sep-pf-6} $(\bm a_i)_j$ does not depend on $j$ either.
For these $i$, let us write $\phi_i \colonequals (\bm a_i)_j$.
Evaluating \eqref{eq:blocksym-sep-pf-2} as a bilinear form on $\bv_j$ and $\bv_k$, we then find
\begin{equation}
    X_{jk} = \la \bv_j, \bv_k \ra = \sum_{i: \eta_i \neq 0} \phi_i \la \bb_i, \bv_j \ra \la \bb_i, \bv_k \ra = \sum_{i:\eta_i \neq 0} \phi_i \eta_i \sgn(\la \bb_i, \bv_j \ra)\sgn(\la \bb_i, \bv_k \ra).
\end{equation}
When $\eta_i \neq 0$, then $\phi_i\eta_i = \kappa_i$, and when $\eta_i = 0$ then $\kappa_i = 0$.
Therefore, we have in fact
\begin{equation}
    X_{jk} = \sum_{i = 1}^m \kappa_i \sgn(\la \bb_i, \bv_j \ra)\sgn(\la \bb_i, \bv_k \ra),
\end{equation}
showing $\bX \in \fC^N$.

The converse is simpler: suppose that $\bX \in \fC^N$ and $\bX = \Gram(\bv_1, \dots, \bv_N) \in \fC^N$ for $\bv_1, \dots, \bv_N \in \RR^r$.
Let $\bv \in \RR^{rN}$ be the concatenation of the $\bv_1, \dots, \bv_N$.
We will build $\bM \in \sB_{\sep}(N, r)$ by essentially reversing the process described in the proof of Proposition~\ref{prop:blocksym-rank-constrained}.
Let $\rho_1, \dots, \rho_m \geq 0$ with $\sum_{i = 1}^m \rho_i = 1$ and $\tilde{\bd}_1, \dots, \tilde{\bd}_m \in \{\pm 1\}^N$ be such that
\begin{equation}
    \label{eq:blocksym-sep-pf-7}
    \bX = \sum_{k = 1}^m \rho_k \tilde{\bd}_k\tilde{\bd}_k^\top.
\end{equation}
We may assume without loss of generality that $m \geq r$, by adding extra terms with zero coefficient to this expression.
Then, writing $\bd_i \colonequals ((\tilde{\bd}_k)_i)_{k = 1}^m \in \RR^m$, $\bR = \diag(\bm\rho)$, and $\bv_i^\prime = \bR^{1/2}\bd_i$, \eqref{eq:blocksym-sep-pf-7} implies that $\bX = \Gram(\bv_1^\prime, \dots, \bv_N^\prime)$.
There then exists $\bZ \in \RR^{m \times r}$ such that $\bZ\bv_i = \bv_i^\prime$ and $\bZ^\top \bZ = \bm I_r$.

We let $\bD_i \colonequals \diag(\bd_i)$, and define $\bM \in \RR^{rN \times rN}$ to have blocks
\begin{equation}
    \bM_{[ij]} \colonequals \bZ^\top \bD_i\bD_j \bZ = (\bD_i\bZ)^\top (\bD_j \bZ).
    \label{eq:blocksym-sep-pf-8}
\end{equation}
The last expression gives $\bM$ as a Gram matrix, so $\bM \succeq \bm 0$.
Since $\bD_i^2 = \bm I_r$ for each $i \in [N]$, $\bM_{[ii]} = \bm I_r$, and since $\bD_1, \dots, \bD_N$ commute, $\bM_{[ij]}$ is symmetric.
Thus, $\bM \in \sB(N, r)$.
We also have
\begin{equation}
\bv^\top \bM \bv = \sum_{i = 1}^N \sum_{j = 1}^N \bv_i^{\prime^\top} \bD_i \bD_j \bv_j^{\prime^\top} = \sum_{i = 1}^N \sum_{j = 1}^N \bd_i^\top \bR^{1/2} \bD_i \bD_j \bR^{1/2} \bd_j = \sum_{i = 1}^N \sum_{j = 1}^N \sum_{k = 1}^m \rho_k = N^2.
\end{equation}
It only remains to check that $\bM$ is separable.
To do this, let $\bz_1, \dots, \bz_m \in \RR^r$ be the rows of $\bZ$, then it is straightforward to check against \eqref{eq:blocksym-sep-pf-8} that we can write $\bM = \sum_{i = 1}^m (\tilde{\bd}_i \otimes \bz_i)(\tilde{\bd}_i \otimes \bz_i)^\top$.

\section{Proofs of Results on Partial Transposition}
\label{app:partial-transpose}

\subsection{Proof of Proposition~\ref{prop:schmidt-decomp}}

This result is simply a matter of applying the vectorization operation $\vec$ to the singular value decomposition: if $\bV = \sum_{i = 1}^r \sigma_i \by_i \bz_i^\top$ for $\by_i \in \RR^r$ and $\bz_i \in \RR^N$, then, noting that $\vec(\by_i\bz_i^\top) = \bz_i \otimes \by_i$ and $\vec: \RR^{r \times N} \to \RR^{rN}$ is linear, the result follows.

\subsection{Proof of Proposition~\ref{prop:rank-one-pt}}

Suppose $\bV \in \RR^{r \times N}$ with $r \leq N$ has singular value decomposition $\bV = \sum_{i = 1}^r \sigma_i\by_i \bz_i^\top$ for orthonormal sets of $\by_i \in \RR^r$ and $\bz_i \in \RR^N$ and with $\sigma_i \geq 0$.
Let $\bv = \vec(\bV)$.
Applying Proposition~\ref{prop:schmidt-decomp}, we may write
\begin{align}
  \bv\bv^\top
  &= \left(\sum_{i = 1}^r \sigma_i \bz_i \otimes \by_i\right) \left(\sum_{i = 1}^r \sigma_i \bz_i \otimes \by_i \right)^\top \nonumber \\
  &= \sum_{i = 1}^r \sum_{j = 1}^r \sigma_i\sigma_j (\bz_i\bz_j^\top) \otimes (\by_i \by_j^\top) \nonumber \\
  &= \sum_{i = 1}^r \sum_{j = 1}^r \sigma_i\sigma_j (\bz_i\otimes \by_i) (\bz_j \otimes \by_j)^\top.
\end{align}
Therefore, the partial transpose is
\begin{align}
  (\bv\bv^\top)^{\ptop}
  &= \sum_{i = 1}^r \sum_{j = 1}^r \sigma_i\sigma_j (\bz_i\bz_j^\top) \otimes (\by_j \by_i^\top) \nonumber \\
  &= \sum_{i = 1}^r \sum_{j = 1}^r \sigma_i\sigma_j (\bz_i \otimes \by_j ) \otimes (\bz_j \otimes \by_i)^\top \nonumber \\
  &= \sum_{i = 1}^r \sigma_i^2 (\bz_i \otimes \by_i)(\bz_i \otimes \by_i)^\top \nonumber \\
  &\hspace{2cm}+ \sum_{1 \leq i < j \leq r} \sigma_i\sigma_j \bigg((\bz_i \otimes \by_j)(\bz_j \otimes \by_i)^\top + (\bz_j \otimes \by_i)(\bz_i \otimes \by_j)^\top\bigg),
    \label{eq:pt-pf-1}
\end{align}
and the result follows by diagonalizing the rank-two matrices in the second sum.

\subsection{Proof of Proposition~\ref{prop:v-sym}}

Suppose $\bV \in \RR^{r \times N}$ with $r \leq N$ has full rank and singular value decomposition $\bV = \sum_{i = 1}^r \sigma_i\by_i \bz_i^\top$ for an orthonormal basis of $\by_i \in \RR^r$ and an orthonormal set $\bz_i \in \RR^N$, with $\sigma_i > 0$ by the full-rank condition.
Let $\bv = \vec(\bV)$.
Let us also extend $\bz_1, \dots, \bz_r$ with $\bz_{r + 1}, \dots, \bz_N$ to a full orthonormal basis.

Since $\bV\bV^\top = \sum_{i = 1}^r \sigma_i^2 \by_i\by_i^\top$, we may expand
\begin{equation}
    \bm I_N \otimes (\bV\bV^\top) = \left(\sum_{i = 1}^N \bz_i\bz_i^\top\right) \otimes \left(\sum_{j = 1}^r \sigma_j^2 \by_j\by_j^\top\right) = \sum_{i = 1}^N \sum_{j = 1}^r \sigma_j^2 (\bz_i \otimes \by_j)(\bz_i \otimes \by_j)^\top.
\end{equation}
Dividing this sum into those summands with $i \leq r$ and those with $i > r$ and subtracting \eqref{eq:pt-pf-1}, we may write
\begin{align}
  \bm I_N \otimes (\bV\bV^\top)
  &- (\bv\bv^\top)^{\ptop} \nonumber \\
  &= \sum_{1 \leq i < j \leq r} \Bigg(\frac{1}{2}\sigma_i^2 (\bz_j \otimes \by_i)(\bz_j \otimes \by_i)^\top + \frac{1}{2}\sigma_j^2 (\bz_i \otimes \by_j)(\bz_i \otimes \by_j)^\top \nonumber \\
  &\hspace{2cm}- \sigma_i\sigma_j(\bz_i \otimes \by_j)(\bz_j \otimes \by_i)^\top - \sigma_i\sigma_j(\bz_j \otimes \by_i)(\bz_i \otimes \by_j)^\top\Bigg) \nonumber \\
  &\hspace{0.8cm}+ \sum_{i = r + 1}^N \sum_{j = 1}^r \sigma_j^2 (\bz_i \otimes \by_j)(\bz_i \otimes \by_j)^\top \nonumber \\
  &= \frac{1}{2}\sum_{1 \leq i < j \leq r} \left(\sigma_i\bz_j \otimes \by_i - \sigma_j\bz_i \otimes \by_j\right)\left(\sigma_i\bz_j \otimes \by_i - \sigma_j\bz_i \otimes \by_j\right)^\top \nonumber \\
  &\hspace{0.8cm}+ \sum_{i = r + 1}^N \sum_{j = 1}^r \sigma_j^2 (\bz_i \otimes \by_j)(\bz_i \otimes \by_j)^\top. \label{eq:v-sym-pf-1}
\end{align}
We thus find an alternate proof that $\bm I_N \otimes (\bV\bV^\top) - (\bv\bv^\top)^{\ptop} \succeq \bm 0$.
The benefit of this approach is that it allows us to read off the subspace we are interested in directly: note that up to rescaling the expression \eqref{eq:v-sym-pf-1} is a spectral decomposition, and thus
\begin{align}
  \ker\bigg(\bm I_N
  &\otimes (\bV\bV^\top) - (\bv\bv^\top)^{\ptop}\bigg)^\perp \nonumber \\
  &= \mathsf{span}\left(\left\{\frac{1}{\sqrt{\sigma_i^2 + \sigma_j^2}}\left(\sigma_i\bz_j \otimes \by_i - \sigma_j\bz_i \otimes \by_j\right)\right\}_{1 \leq i < j \leq r} \cup \{\bz_i \otimes \by_j\}_{i \in [N] \setminus [r], j \in [r]}\right), \\
  \ker\bigg(\bm I_N &\otimes (\bV\bV^\top) - (\bv\bv^\top)^{\ptop}\bigg) \nonumber \\
  &= \mathsf{span}\left(\left\{\frac{1}{\sqrt{\sigma_i^2 + \sigma_j^2}}\left(\sigma_i\bz_j \otimes \by_i + \sigma_j\bz_i \otimes \by_j\right)\right\}_{1 \leq i < j \leq r} \cup \{\bz_i \otimes \by_i\}_{i \in [r]}\right), \label{eq:v-sym-pf-2}
\end{align}
where the first equality follows from \eqref{eq:v-sym-pf-1} and the second may be checked by counting dimensions and verifying mutual orthogonalities.
It is also straightforward to verify that the vectors enumerated in \eqref{eq:v-sym-pf-2} are orthonormal, and thus give an orthonormal basis for  $\ker(\bm I_N \otimes (\bV\bV^\top) - (\bv\bv^\top)^{\ptop})$.

The only remaining task is to check the alternate description
\begin{equation}
    \ker\bigg(\bm I_N \otimes (\bV\bV^\top) - (\bv\bv^\top)^{\ptop}\bigg) \stackrel{?}{=} \left\{ \vec(\bS \bV): \bS \in \RR^{r \times r}_{\sym} \right\} \equalscolon V_{\sym}.
    \label{eq:v-sym-pf-3}
\end{equation}
We have $\dim(\ker\bigg(\bm I_N \otimes (\bV\bV^\top) - (\bv\bv^\top)^{\ptop}\bigg)) = \frac{r(r + 1)}{2}$ by \eqref{eq:v-sym-pf-2}.
Since $\bv_i$ are a spanning set, if $\vec(\bS\bV) = \bm 0$ then $\bS = \bm 0$, so the map $\bS \mapsto \vec(\bS\bV)$ is injective and thus $\dim(V_{\sym}) = \dim(\RR^{r \times r}_{\sym}) = \frac{r(r + 1)}{2}$ as well.
Therefore, to show \eqref{eq:v-sym-pf-3} it suffices to show one inclusion.

Suppose that $\bS \in \RR^{r \times r}_{\sym}$, then
\begin{align}
  ((\bm I_N \otimes (\bV\bV^\top) - (\bv\bv^\top)^{\ptop})\vec(\bS\bV))_{[i]}
  &= (\bV\bV^\top)\bS\bv_i - \sum_{j = 1}^N \bv_j\bv_i^\top \bS\bv_j \nonumber \\
  &= \sum_{j = 1}^N \bv_j \bv_j^\top \bS \bv_i - \sum_{j = 1}^N \bv_j\bv_i^\top \bS\bv_j \nonumber \\
  &= \bm 0,
\end{align}
where in the last step we use that $\bS$ is symmetric.
Thus, $\vec(\bS\bV) \in \ker(\bm I_N \otimes (\bV\bV^\top) - (\bv\bv^\top)^{\ptop})$, so $V_{\sym} \subseteq \ker(\bm I_N \otimes (\bV\bV^\top) - (\bv\bv^\top)^{\ptop})$, which completes the proof by the previous dimension counting argument.

\section{Tight Frame Projector Calculations}
\label{app:projectors}

In this appendix we will derive formulae for the orthogonal projectors to various subspaces associated with a UNTF $\bv_1, \dots, \bv_N \in \SS^{r - 1}$, and the specializations to the case of ETFs.
Let $\bV \in \RR^{r \times N}$ have the $\bv_i$ as its columns, then recall that we define a map $\sV: \RR^{r \times r}_{\sym} \to \RR^{rN}$ by $\sV(\bS) = \sqrt{\frac{r}{N}}\vec(\bS\bV)$, which by Proposition~\ref{prop:lin-isom} is a linear isometric embedding.
Let us also write $\bX = \bV^\top \bV \in \RR^{N \times N}$ for the Gram matrix.

We then are interested in the projector to the following subspace:
\begin{equation}
  V_{\sym}^\prime \colonequals \sV\left(\left\{ \bS \in \RR^{r \times r}_{\sym}: \la \bS, \bv_i\bv_i^\top \ra = 0 \text{ for } i \in [N]\right\}\right).
\end{equation}
As a warmup, we will also consider the following simpler subspace:
\begin{equation}
    V_{\sym} \colonequals \sV(\RR^{r \times r}_{\sym}).
\end{equation}
The idea of the calculation in both cases will be as follows: suppose $V \subset \RR^{r \times r}_{\sym}$ is some subspace and $\by \in \RR^{rN}$ is the concatenation of $\by_1, \dots, \by_N \in \RR^r$.
Then by the variational characterization of the orthogonal projector, $\bP_{\sV(V)}\by = \sV(\bS^\star(\by))$, where
\begin{align}
  \mathsf{obj}(\bS; \by)
  &\colonequals \frac{1}{2}\sum_{i = 1}^N \left\|\sqrt{\frac{r}{N}}\bS\bv_i - \by_i\right\|_2^2 \nonumber \\
  &= \frac{1}{2}\|\by\|_2^2 + \frac{1}{2}\Tr(\bS^2) - \left\la \bS, \sqrt{\frac{r}{N}}\sum_{i = 1}^N\frac{\bv_i\by_i^\top + \by_i\bv_i^\top}{2}\right\ra, \label{eq:var-sym-proj-obj} \\
    \bS^\star(\by) &= \argmin_{\bS \in V} \mathsf{obj}(\bS; \by). \label{eq:var-sym-proj}
\end{align}
a minimization which we will solve by introducing Lagrange multipliers for the constraint $\bS \in V$, which will reduce the task to solving a linear system in the Lagrange multiplier variables.\footnote{Note that the simple form of the quadratic term in $\bS$ is a consequence of the $\bv_i$ forming a UNTF, whereby $\sum_{i = 1}^N \bv_i\bv_i^\top = \bV\bV^\top = \frac{N}{r} \bm I_r$.
  In a more general setting, the matrix $\bV\bV^\top$ would appear and, upon differentiating with respect to $\bS$, we would not get a formula for the optimizer $\bS^\star$ but rather a so-called \emph{continuous matrix Lyapunov equation} $(\bV\bV^\top)\bS^\star + \bS^\star(\bV\bV^\top) = \bQ$.
  Such an equation in principle admits an analytic solution by reducing to a linear equation in $\vec(\bS^\star)$ (see e.g.\ \cite{kuvcera:74}), but this would further complicate the calculations.}

\subsection{Projector to $V_{\sym}$}

We illustrate the first part of this idea below for the simplest case of $V = V_{\sym}$, where no Lagrange multipliers are required.

\begin{proposition}
    $\bP_{V_\sym} = \frac{r}{2N}(\bX \otimes \bm I_r + (\bv\bv^\top)^{\ptop})$, where $\bX = \bV^\top \bV$ is the Gram matrix of the $\bv_i$.
\end{proposition}
\begin{proof}
    We will solve the variational description \eqref{eq:var-sym-proj} with $V = V_{\sym}$.
    Since the optimization is unconstrained, we may compute directly the first-order condition for the optimizer
    \begin{equation}
        0 = \frac{\partial \mathsf{obj}}{\partial \bS}(\bS^\star; \by) = \bS^\star - \sqrt{\frac{r}{N}}\sum_{j = 1}^N\frac{\bv_j\by_j^\top + \by_j\bv_j^\top}{2},
    \end{equation}
    thus the optimizer is
    \begin{equation}
        \bS^\star = \bS^\star(\by) = \sqrt{\frac{r}{N}}\sum_{j = 1}^N\frac{\bv_j\by_j^\top + \by_j\bv_j^\top}{2}.
    \end{equation}
    Then, the blocks of the projection of $\by$ may be recovered as
    \begin{align}
      (\bP_{V_\sym}\by)_{[i]}
      &= (\sV(\bS^\star))_{[i]} \nonumber \\
      &= \sqrt{\frac{r}{N}}\bS^\star\bv_i \nonumber \\
      &= \frac{r}{2N}\sum_{j = 1}^N\left(\la \by_j, \bv_i \ra\bv_j + \la \bv_i, \bv_j \ra \by_j\right) \nonumber \\
      &= \sum_{j = 1}^N\frac{r}{2N}\left(\la \bv_i, \bv_j \ra \bm I_r + \bv_j\bv_i^\top\right) \by_j.
    \end{align}
    In particular, the matrices in each term of the sum give the blocks $(\bP_{V_\sym})_{[ij]}$, whereby the formula in the statement is clearly correct for each block.
\end{proof}

\subsection{Projector to $V_{\sym}^\prime$}

The case of $V = V_{\sym}^\prime$, being a constrained optimization, requires more intermediate calculations in order to determine the Lagrange multipliers.
An important role will be played by the Hadamard square of the Gram matrix, $\bX^{\odot 2}$, which is equivalently the Gram matrix of the matrices $\bv_i\bv_i^\top$ (and which figured in the proof of the Gerzon bound, Proposition~\ref{prop:etf-gerzon-bound}, and the description of perturbations of matrices in the elliptope, Proposition~\ref{prop:li-tam}).
To perform our calculations in closed form, we will need to compute the inverse of this matrix explicitly.
We will first give a general result in terms of this inverse and then show how the inverse may be computed for the case of ETFs needed in the main text.

\begin{proposition}
    \label{prop:v-sym-prime}
    Suppose that the matrices $\bv_i\bv_i^\top$ are linearly independent, or equivalently that the matrix $\bX^{\odot 2}$ is non-singular.
    Then, the blocks of $\bP_{V_{\sym}^\prime}$ are given by
    \begin{equation}
        (\bP_{V_{\sym}^\prime})_{[ij]} = \frac{r}{N}\left(\frac{1}{2}\la \bv_i, \bv_j \ra \bm I_r + \frac{1}{2}\bv_j \bv_i^\top - \sum_{k = 1}^N \sum_{\ell = 1}^N ((\bX^{\odot 2})^{-1})_{k\ell} \la \bv_i, \bv_k \ra \la \bv_j, \bv_\ell \ra \bv_k\bv_\ell^\top\right).
    \end{equation}
\end{proposition}
\begin{proof}
    We must now solve \eqref{eq:var-sym-proj} with $T = V_{\sym}^\prime = \mathsf{span}(\{\bv_i\bv_i^\top: i \in [N]\})^\perp$.
    We introduce the Lagrangian
    \begin{equation}
        L(\bS, \bm \gamma; \by) \colonequals \mathsf{obj}(\bS; \by) - \left\la \bS, \sum_{i = 1}^N \gamma_i \bv_i\bv_i^\top \right\ra
    \end{equation}
    and write the first-order condition $\frac{\partial L}{\partial \bS}(\bS^\star, \bm\gamma; \by) = 0$, which gives
    \begin{equation}
        \bS^\star = \bS^\star(\by) = \sqrt{\frac{r}{N}}\sum_{j = 1}^N\frac{\bv_j\by_j^\top + \by_j\bv_j^\top}{2} + \sum_{j = 1}^N \gamma_j \bv_j\bv_j^\top.
        \label{eq:v-sym-prime-1}
    \end{equation}
    The other first-order condition $\frac{\partial L}{\partial \bm\gamma}(\bS^\star, \bm\gamma; \by) = 0$ is equivalent to the constraints, $\la \bS^\star, \bv_i\bv_i^\top \ra = 0$ for all $i \in [N]$, which yields the system of linear equations for $\bm\gamma$,
    \begin{equation}
        \sum_{j = 1}^N (\bX^{\odot 2})_{ij} \gamma_j = -\sqrt{\frac{r}{N}}\sum_{j = 1}^N\la \bv_i, \bv_j \ra \la \bv_i, \by_j \ra \text{ for } i \in [N].
    \end{equation}
    Since $\bX^{\odot 2}$ is invertible by assumption, this admits a unique solution which is given by
    \begin{equation}
        \gamma_j = -\sqrt{\frac{r}{N}}\sum_{k = 1}^N \sum_{\ell = 1}^N ((\bX^{\odot 2})^{-1})_{jk}\la \bv_k, \bv_\ell \ra \la \bv_k, \by_\ell \ra.
    \end{equation}
    Substituting into \eqref{eq:v-sym-prime-1}, we find
    \begin{equation}
        \bS^\star = \sqrt{\frac{r}{N}}\left(\sum_{j = 1}^N\frac{\bv_j\by_j^\top + \by_j\bv_j^\top}{2} - \sum_{j = 1}^N \sum_{k = 1}^N \sum_{\ell = 1}^N ((\bX^{\odot 2})^{-1})_{jk}\la \bv_k, \bv_\ell \ra \la \bv_k, \by_\ell \ra \bv_j\bv_j^\top\right).
    \end{equation}
    As before, we recover the blocks of the projection of $\by$,
    \begin{align}
      (\bP_{V_\sym^\prime}\by)_{[i]}
      &= (\sV(\bS^\star))_{[i]} \nonumber \\
      &= \sqrt{\frac{r}{N}}\bS^\star\bv_i \nonumber \\
      &= \frac{r}{N}\left(\sum_{j = 1}^N\frac{\la \bv_i, \by_j \ra\bv_j + \la \bv_i, \bv_j \ra\by_j}{2} - \sum_{j = 1}^N \sum_{k = 1}^N \sum_{\ell = 1}^N ((\bX^{\odot 2})^{-1})_{jk}\la \bv_i, \bv_j \ra \la \bv_k, \bv_\ell \ra \la \bv_k, \by_\ell \ra \bv_j\right) \nonumber \\
      &= \sum_{j = 1}^N\frac{r}{N}\left(\frac{1}{2}\la \bv_i, \bv_j \ra \bm I_r + \frac{1}{2}\bv_j \bv_i^\top - \sum_{k = 1}^N \sum_{\ell = 1}^N ((\bX^{\odot 2})^{-1})_{k\ell} \la \bv_i, \bv_k \ra \la \bv_j, \bv_\ell \ra \bv_k\bv_\ell^\top\right) \by_j,
    \end{align}
    and the result follows.
\end{proof}

\begin{corollary}
    \label{cor:v-sym-prime-proj-etf}
    Suppose that $\bv_1, \dots, \bv_N$ form an ETF with $r > 1$.
    Then, the blocks of $\bP_{V_{\sym}^\prime}$ are given by
    \begin{equation}
        (\bP_{V_{\sym}^\prime})_{[ij]} = \frac{N - r}{N(r - 1)}\bv_i\bv_j^\top + \frac{r}{2N}\bv_j\bv_i^\top + \frac{r}{2N}\la \bv_i, \bv_j \ra \bm I_r - \frac{r^2(N - 1)}{N^2(r - 1)}\sum_{k = 1}^N \la \bv_i, \bv_k \ra \la \bv_j, \bv_k \ra \bv_k\bv_k^\top.
    \end{equation}
\end{corollary}
\begin{proof}
    By Proposition~\ref{prop:etf-gerzon-bound}, the conditions of Proposition~\ref{prop:v-sym-prime} are satisfied, so it suffices to compute $(\bX^{\odot 2})^{-1}$.
    The off-diagonal entries of $\bX$ all equal the coherence $\alpha$, which by Proposition~\ref{prop:etf-welch-bound} is given by
    \begin{equation}
        \alpha = \sqrt{\frac{N - r}{r(N - 1)}}.
    \end{equation}
    Thus, we have
    \begin{equation}
        \bX^{\odot 2} = (1 - \alpha^2) \bm I_N + \alpha^2 \one\one^\top = \frac{N(r - 1)}{r(N - 1)}\bm I_N + \frac{N - r}{r(N - 1)}\one\one^\top.
    \end{equation}
    This matrix may be inverted by the Sherman-Morrison formula, giving
    \begin{equation}
        (\bX^{\odot 2})^{-1} = \frac{r(N - 1)}{N(r - 1)}\bm I_N - \frac{r(N - r)}{N^2(r - 1)}\one\one^\top.
    \end{equation}
    Thus, the entries are
    \begin{equation}
        (\bX^{\odot 2})^{-1}_{ij} = \left\{\begin{array}{lcr} a \colonequals \frac{r((N - 1)^2 + r - 1)}{N^2(r - 1)} & : & i = j, \\ b \colonequals - \frac{r(N - r)}{N^2(r - 1)} & : & i \neq j.\end{array}\right.
    \end{equation}
    Substituting into the expression from Proposition~\ref{prop:v-sym-prime}, we find
    \begin{align}
      \sum_{k = 1}^N \sum_{\ell = 1}^N &((\bX^{\odot 2})^{-1})_{k\ell} \la \bv_i, \bv_k \ra \la \bv_j, \bv_\ell \ra \bv_k\bv_\ell^\top \nonumber \\
      &= (a - b)\sum_{k = 1}^N \la \bv_i, \bv_k \ra \la \bv_j, \bv_k \ra \bv_k\bv_k^\top + b\sum_{k = 1}^N\sum_{\ell = 1}^N \la \bv_i, \bv_k \ra \la \bv_j, \bv_\ell \ra \bv_k\bv_\ell^\top \nonumber \\
      &= (a - b)\sum_{k = 1}^N \la \bv_i, \bv_k \ra \la \bv_j, \bv_k \ra \bv_k\bv_k^\top + b(\bV\bV^\top \bv_i)(\bV\bV^\top \bv_j)^\top \nonumber \\
      &= \frac{r(N - 1)}{N(r - 1)}\sum_{k = 1}^N \la \bv_i, \bv_k \ra \la \bv_j, \bv_k \ra \bv_k\bv_k^\top - \frac{N - r}{r(r - 1)}\bv_i\bv_j^\top.
    \end{align}
    Combining with the full result of Proposition~\ref{prop:v-sym-prime} then gives the claim.
\end{proof}

\end{document}